\documentclass[11pt]{amsart}
\usepackage{amsmath,amsthm,amsfonts,amssymb,latexsym,enumerate}

\begin{document}

\theoremstyle{plain}

\newtheorem{thm}{Theorem}[section]
\newtheorem{lem}[thm]{Lemma}
\newtheorem{pro}[thm]{Proposition}
\newtheorem{cor}[thm]{Corollary}
\newtheorem{que}[thm]{Question}
\newtheorem{rem}[thm]{Remark}
\newtheorem{defi}[thm]{Definition}

\newtheorem*{thmA}{Theorem A}
\newtheorem*{thmB}{Theorem B}
\newtheorem*{thmC}{Theorem C}
\newtheorem*{thmD}{Theorem D}
\newtheorem*{thmE}{Theorem E}
\newtheorem*{corF}{Corollary F}

\newtheorem*{thmAcl}{Main Theorem$^{*}$}
\newtheorem*{thmBcl}{Theorem B$^{*}$}

\numberwithin{equation}{section}

\newcommand{\Maxn}{\operatorname{Max_{\textbf{N}}}}
\newcommand{\Syl}{\operatorname{Syl}}
\newcommand{\dl}{\operatorname{dl}}
\newcommand{\Con}{\operatorname{Con}}
\newcommand{\cl}{\operatorname{cl}}
\newcommand{\Stab}{\operatorname{Stab}}
\newcommand{\Aut}{\operatorname{Aut}}
\newcommand{\Ker}{\operatorname{Ker}}
\newcommand{\fl}{\operatorname{fl}}
\newcommand{\Irr}{\operatorname{Irr}}
\newcommand{\SL}{\operatorname{SL}}
\newcommand{\FF}{\mathbb{F}}
\newcommand{\NN}{\mathbb{N}}
\newcommand{\N}{\mathbf{N}}
\newcommand{\C}{\mathbf{C}}
\newcommand{\OO}{\mathbf{O}}
\newcommand{\F}{\mathbf{F}}

\renewcommand{\labelenumi}{\upshape (\roman{enumi})}

\newcommand{\PSL}{\operatorname{PSL}}
\newcommand{\PSU}{\operatorname{PSU}}

\providecommand{\V}{\mathrm{V}}
\providecommand{\E}{\mathrm{E}}
\providecommand{\ir}{\mathrm{Irr_{rv}}}
\providecommand{\Irrr}{\mathrm{Irr_{rv}}}
\providecommand{\re}{\mathrm{Re}}

\def\irrp#1{{\rm Irr}_{p'}(#1)}

\def\Z{{\mathbb Z}}
\def\C{{\mathbb C}}
\def\Q{{\mathbb Q}}
\def\irr#1{{\rm Irr}(#1)}
\def\irrv#1{{\rm Irr}_{\rm rv}(#1)}
\def \c#1{{\cal #1}}
\def\cent#1#2{{\bf C}_{#1}(#2)}
\def\syl#1#2{{\rm Syl}_#1(#2)}
\def\nor{\triangleleft\,}
\def\oh#1#2{{\bf O}_{#1}(#2)}
\def\Oh#1#2{{\bf O}^{#1}(#2)}
\def\zent#1{{\bf Z}(#1)}
\def\det#1{{\rm det}(#1)}
\def\ker#1{{\rm ker}(#1)}
\def\norm#1#2{{\bf N}_{#1}(#2)}
\def\alt#1{{\rm Alt}(#1)}
\def\iitem#1{\goodbreak\par\noindent{\bf #1}}
   \def \mod#1{\, {\rm mod} \, #1 \, }
\def\sbs{\subseteq}

\newcommand{\fS}{{\mathfrak{S}}}
\def\gc{{\bf GC}}
\def\ngc{{non-{\bf GC}}}
\def\ngcs{{non-{\bf GC}$^*$}}
\newcommand{\notd}{{\!\not{|}}}
\newcommand{\Out}{{\mathrm {Out}}}
\newcommand{\Mult}{{\mathrm {Mult}}}
\newcommand{\Inn}{{\mathrm {Inn}}}
\newcommand{\IBR}{{\mathrm {IBr}}}
\newcommand{\IBRL}{{\mathrm {IBr}}_{\ell}}
\newcommand{\IBRP}{{\mathrm {IBr}}_{p}}
\newcommand{\ord}{{\mathrm {ord}}}
\def\id{\mathop{\mathrm{ id}}\nolimits}
\renewcommand{\Im}{{\mathrm {Im}}}
\newcommand{\Ind}{{\mathrm {Ind}}}
\newcommand{\diag}{{\mathrm {diag}}}
\newcommand{\soc}{{\mathrm {soc}}}
\newcommand{\End}{{\mathrm {End}}}
\newcommand{\sol}{{\mathrm {sol}}}
\newcommand{\Hom}{{\mathrm {Hom}}}
\newcommand{\Mor}{{\mathrm {Mor}}}
\newcommand{\Mat}{{\mathrm {Mat}}}
\def\rank{\mathop{\mathrm{ rank}}\nolimits}
\newcommand{\Tr}{{\mathrm {Tr}}}
\newcommand{\tr}{{\mathrm {tr}}}
\newcommand{\Gal}{{\mathrm {Gal}}}
\newcommand{\Spec}{{\mathrm {Spec}}}
\newcommand{\ad}{{\mathrm {ad}}}
\newcommand{\Sym}{{\mathrm {Sym}}}
\newcommand{\Char}{{\mathrm {char}}}
\newcommand{\pr}{{\mathrm {pr}}}
\newcommand{\rad}{{\mathrm {rad}}}
\newcommand{\abel}{{\mathrm {abel}}}
\newcommand{\codim}{{\mathrm {codim}}}
\newcommand{\ind}{{\mathrm {ind}}}
\newcommand{\Res}{{\mathrm {Res}}}
\newcommand{\Ann}{{\mathrm {Ann}}}
\newcommand{\Ext}{{\mathrm {Ext}}}
\newcommand{\Alt}{{\mathrm {Alt}}}
\newcommand{\AAA}{{\sf A}}
\newcommand{\SSS}{{\sf S}}
\newcommand{\CC}{{\mathbb C}}
\newcommand{\CB}{{\mathbf C}}
\newcommand{\RR}{{\mathbb R}}
\newcommand{\QQ}{{\mathbb Q}}
\newcommand{\ZZ}{{\mathbb Z}}
\newcommand{\NB}{{\mathbf N}}
\newcommand{\ZB}{{\mathbf Z}}
\newcommand{\OB}{{\mathbf O}}
\newcommand{\EE}{{\mathbb E}}
\newcommand{\PP}{{\mathbb P}}
\newcommand{\GC}{{\mathcal G}}
\newcommand{\HC}{{\mathcal H}}
\newcommand{\PC}{{\mathcal P}}
\newcommand{\QC}{{\mathcal Q}}
\newcommand{\MC}{{\mathcal M}}
\newcommand{\TC}{{\mathcal T}}
\newcommand{\LC}{{\mathcal L}}
\newcommand{\CL}{{\mathcal C}}
\newcommand{\EC}{{\mathcal E}}
\newcommand{\GCD}{\GC^{*}}
\newcommand{\TCD}{\TC^{*}}
\newcommand{\tV}{\tilde{V}}
\newcommand{\bFq}{{\bar{\FF}_q}}
\newcommand{\FD}{F^{*}}
\newcommand{\GD}{G^{*}}
\newcommand{\HD}{H^{*}}
\newcommand{\GCF}{\GC^{F}}
\newcommand{\TCF}{\TC^{F}}
\newcommand{\PCF}{\PC^{F}}
\newcommand{\GCDF}{(\GC^{*})^{F^{*}}}
\newcommand{\RGTT}{R^{\GC}_{\TC}(\theta)}
\newcommand{\RGTA}{R^{\GC}_{\TC}(1)}
\newcommand{\SR}{{^*R}}
\newcommand{\Om}{\Omega}
\newcommand{\eps}{\epsilon}
\newcommand{\varep}{\varepsilon}
\newcommand{\al}{\alpha}
\newcommand{\chis}{\chi_{s}}
\newcommand{\sigmad}{\sigma^{*}}
\newcommand{\PA}{\boldsymbol{\alpha}}
\newcommand{\gam}{\gamma}
\newcommand{\lam}{\lambda}
\newcommand{\la}{\langle}
\newcommand{\ra}{\rangle}
\newcommand{\hs}{\hat{s}}
\newcommand{\htt}{\hat{t}}
\newcommand{\hP}{\hat{P}}
\newcommand{\hR}{\hat{R}}
\newcommand{\hG}{\hat{G}}
\newcommand{\bM}{\bar{M}}
\newcommand{\tv}{\tilde\varepsilon}
\newcommand{\tCk}{\tilde{C}_\kappa}
\newcommand{\orm}{{\mathrm {o}}}
\newcommand{\tn}{\hspace{0.5mm}^{t}\hspace*{-0.2mm}}
\newcommand{\ta}{\hspace{0.5mm}^{2}\hspace*{-0.2mm}}
\newcommand{\tb}{\hspace{0.5mm}^{3}\hspace*{-0.2mm}}
\def\skipa{\vspace{-1.5mm} & \vspace{-1.5mm} & \vspace{-1.5mm}\\}
\newcommand{\tw}[1]{{}^#1\!}
\renewcommand{\mod}{\bmod \,}

\marginparsep-0.5cm

\renewcommand{\thefootnote}{\fnsymbol{footnote}}
\footnotesep6.5pt

\title{Restriction of odd degree characters and natural correspondences}

\author[E. Giannelli]{Eugenio Giannelli}
\address{FB Mathematik, TU Kaiserslautern, Postfach 3049,
        67653 Kaisers\-lautern, Germany.}
\email{giannelli@mathematik.uni-kl.de}

\author[A. Kleshchev]{Alexander Kleshchev}
\address{Department of Mathematics 
University of Oregon
Eugene, OR 97403-1222
USA }
\email{klesh@uoregon.edu} 
\author[G. Navarro]{Gabriel Navarro}
\address{Departament d'\`Algebra, Universitat de Val\`encia, 46100 Burjassot,
Val\`encia, Spain}
\email{gabriel.navarro@uv.es}
\author[P. H. Tiep]{Pham Huu Tiep}
\address{Department of Mathematics, University of Arizona, Tucson, AZ 85721, USA}
\email{tiep@math.arizona.edu}

\thanks{The first author gratefully acknowledges financial support by the
  ERC Advanced Grant 291512. The second author was supported by the NSF (grant DMS-1161094) and the Max-Planck-Institut. 
The research of the third author is supported by the Prometeo/Generalitat Valenciana,
Proyectos MTM2013-40464-P. The fourth author
was supported by the NSF (grant DMS-1201374).}

\thanks{We thank Rod Gow for communicating his observation concerning $GL_n(q)$ in Corollary F to us,
as well as for suggesting that the same statement should hold for $GU_n(q)$.}
 


\begin{abstract}
Let $q$ be an odd prime power, $n > 1$, and let $P$ denote a maximal parabolic subgroup of 
$GL_n(q)$ with Levi subgroup $GL_{n-1}(q) \times GL_1(q)$. 
We restrict the odd-degree irreducible characters of $GL_n(q)$ to
$P$ to discover a natural correspondence of characters, both for $GL_n(q)$ and $SL_n(q)$. 
A similar result is established for certain finite groups with self-normalizing Sylow $p$-subgroups.
We also construct a canonical bijection between the odd-degree irreducible characters of $\SSS_n$ 
and those of $M$, where $M$ is any maximal subgroup of $\SSS_n$ of odd index; as well as 
between the odd-degree irreducible characters of $G = GL_n(q)$ or $GU_n(q)$ with $q$ odd and those of 
$N_{G}(P)$, where $P$ is a Sylow $2$-subgroup of $G$.
Since our bijections commute with the action of the absolute Galois group over the
rationals,  we conclude that the fields of values of character
correspondents  are the same. We use this to answer some questions of R. Gow.

\end{abstract}

\maketitle

\section{Introduction}
It is not often the case that a natural correspondence of characters between a group $G$
and a subgroup $H$ of $G$ is found. Even more rarely this correspondence
can be described by inspecting the restriction of characters from $G$ to $H$. The paradigmatic
example of this is the Glauberman correspondence between the $P$-invariant
irreducible characters of a finite group $G$
of order not divisible by a prime $p$, acted by the $p$-group $P$, and the
irreducible characters of the fixed point subgroup $\cent GP$. 

If $G$ is a finite group and
$p$ is a prime, the {\it McKay conjecture} (cf. \cite{McKay,McK}) asserts that $|\Irr_{p'}(G)|=|\Irr_{p'}(\NB_G(P))|$ for $P\in\Syl_p(G)$, where 
$\operatorname{Irr}_{p'}(G)$ is the set of 
the irreducible complex characters of $G$ of degree not divisible
by $p$.

The focus of this paper is on the existence of {\it canonical correspondences} between 
$\Irr_{p'}(G)$ and $\Irr_{p'}(H)$ for certain pairs $(G,H)$ of finite groups with $G > H \geq \NB_G(P)$. Even in the cases where the McKay 
conjecture is known to hold for $G$ and $H$ (cf. \cite{O1}, \cite{MS}) and thereby
the existence of a bijection between $\Irr_{p'}(G)$ and $\Irr_{p'}(H)$ is guaranteed, this
resulting bijection may not give a canonical correspondence.  
For instance, one expects a canonical correspondence 
to commute with the action of the absolute Galois group over the
rationals, and in this case, the fields of values of character
correspondents  must be the same. This does not happen often.
Also, one expects that canonical correspondences
between $\Irr_{p'}(G)$ and $\Irr_{p'}(H)$ will commute with 
every automorphism of $G$ that stabilizes $H$, and  provide 
essential information
on cohomological character theoretical questions.
Furthermore, there is some hope that certain
canonical correspondences will play 
an important role in proving various refinements of the McKay conjecture (eg. \cite{N2}).  

\medskip

 For any character $\chi$ of $G$ we denote by $\chi_H$ its restriction to a subgroup $H$. 
\begin{defi}
{\rm 
An arbitrary subgroup $H\leq G$ is called {\em $p$-restriction good}\, if for every $\chi\in\operatorname{Irr}_{p'}(G)$, there exists a unique $\chi^*\in \operatorname{Irr}_{p'}(H)$ such that  
$\chi_H=\chi^*+\Delta$  
and either $\Delta = 0$ or all irreducible constituents of $\Delta$ have degrees divisible by $p$. A $p$-restriction good subgroup $H\leq G$ is 
called {\em $p$-restriction canonical} if the map $\chi \mapsto \chi^*$ yields a bijection between $\operatorname{Irr}_{p'}(G)$
and $\operatorname{Irr}_{p'}(H)$.
}
\end{defi}

\medskip
Very recently, the following result\footnote{In fact, we independently conjectured this statement. But while we were working
on the proof of it, we learned of the preprint \cite{APS} in which the conjecture was proved.}
has been proved: 

\begin{thm}\label{sn} {\rm \cite{APS}} 
Let $n\in\Z_{>1}$. Then $\SSS_{n-1}$ is a $2$-restriction good subgroup in $\SSS_n$. Moreover, if $n$ is odd, then $\SSS_{n-1}$ is a $2$-restriction canonical subgroup in $\SSS_n$. \end{thm}
 
In this paper, we prove:

\begin{thmA}
Let $n\in\Z_{>1}$, $q$ be an odd prime power, and $P$ be  a maximal parabolic subgroup of 
$GL_n(q)$ with Levi subgroup $GL_{n-1}(q) \times GL_1(q)$. Then 
$P$ is a $2$-restriction good subgroup in $GL_n(q)$. Moreover, if $n$ is odd, then $P$ is a $2$-restriction canonical subgroup in $GL_n(q)$.
\end{thmA}

\begin{thmB}
Let $n\in\Z_{>1}$ be odd, $q$ an odd prime power, and $Q$ be a maximal parabolic subgroup of 
$SL_n(q)$ with Levi subgroup $(GL_{n-1}(q) \times GL_1(q)) \cap SL_n(q)$. Then 
$Q$ is a $2$-restriction canonical subgroup in $SL_n(q)$.
\end{thmB}


\begin{thmC}
Let $G$ be a finite group, $p$ be a prime, and  $P \in \Syl_p(G)$.
Suppose that $P = \norm GP$, and in addition that $G$ is a solvable group if $p=2$. 
Let $P \leq H \le G$. Then $H$ is a $p$-restriction canonical subgroup in $G$.
\end{thmC}

All these theorems might suggest that further results of this type can hold true
for arbitrary finite groups with self-normalizing Sylow $2$-subgroups. However, 
the group $G=SL_3(2)$ has a self-normalizing Sylow $2$-subgroups,
and an irreducible character $\chi \in \irr G$ of degree 7 such that
the restriction of $\chi$ to {\it every} odd-index proper subgroup $H$ of $G$
has exactly three irreducible constituents of odd degree. This example also shows that
Theorem A does not hold when $2|q$. Other examples also show that analogues of Theorem \ref{sn} and 
Theorems A, B do not seem to hold when $p > 2$.
However, canonical character correspondences, although
not necessarily defined by restriction, can be obtained for symmetric groups and finite general linear and unitary 
groups, again for $p=2$. 
The meaning of the word {\em canonical} in the following theorems will become clear in Sections~\ref{SD}, \ref{SE};
in particular, the constructed character correspondences commute with the action of Galois automorphisms and 
group automorphisms.

\begin{thmD}
Let $n \in \Z_{>1}$ and let $M$ be a maximal subgroup of $\SSS_n$ of odd index. Then there is a canonical 
bijection between $\Irr_{2'}(\SSS_n)$ and $\Irr_{2'}(M)$.
\end{thmD}

\begin{thmE}
Let $n \in \Z_{\geq 1}$, $q$ be an odd prime power, $G = GL_n(q)$ or $GU_n(q)$, and $P \in \Syl_2(G)$. Then there is a canonical 
bijection $\chi \mapsto \chi^\sharp$ between $\Irr_{2'}(G)$ and $\Irr_{2'}(\NB_G(P))$.
\end{thmE}

We will prove (see Theorem \ref{main-glu}) that
our canonical bijection in Theorem E
commutes with Galois action, and this will imply, for instance, that the fields of values of the odd-degree
 irreducible characters of $GL_n(q)$ and $GU_n(q)$, if $q$ is odd, are in bijection with the fields
of values of the odd-degree characters of $\NB_G(P)$ for $P \in \Syl_2(G)$. This does not happen in 
$GL_2(4)$ or $GL_2(8)$. Two other cases where there exist a canonical
correspondence for the McKay conjecture for $p=2$  are in solvable groups
\cite{Isa73}  and symmetric groups \cite{G}, see also Theorem \ref{sym}.
To illustrate the power of canonical maps, we can answer a question
of Gow, which was privately communicated to us.

\begin{corF}
The number of real-valued, irreducible characters of odd degree of $G = GL_n(q)$ and $GU_n(q)$, with $q$ any odd prime 
power, is equal to that of $\NB_G(P)$ for $P \in \Syl_2(G)$, which is 
$2^{n_1+n_2+ \ldots +n_r+r}$ if $n = 2^{n_1} + 2^{n_2} + \ldots +2^{n_r}$ is the $2$-adic decomposition of $n$. Furthermore,
 all such characters are rational-valued. 
\end{corF}

\section{Restriction to a maximal parabolic subgroup}
Unless otherwise stated, we always assume that $p$ is a prime and $H$ is a subgroup of a finite group $G$. 
We begin with some simple observations.
Note that if the $p$-restriction good subgroup $H$ of $G$ satisfies $|\Irr_{p'}(G)| = |\Irr_{p'}(H)|$, then it is $p$-restriction canonical, by  the following lemma.
(The lemma also indicates a possible weakening of the notion of $p$-restriction good subgroups when one allows a multiplicity $>1$ of the 
$p'$-degree irreducible constituent).

\begin{lem}\label{trivial}
Let $H$ have $p'$-index in $G$ and $|\Irr_{p'}(G)| = |\Irr_{p'}(H)|$. Suppose that for every $\chi \in \Irr_{p'}(G)$, among the irreducible constituents of $\chi_H$ 
there is only one (but possibly with multiplicity $>1$), denoted by $\chi^*$, that has $p'$-degree.
Then the map $*:\chi \mapsto \chi^*$ is a 
bijection between $\Irr_{p'}(G)$ and $\Irr_{p'}(H)$.
\end{lem}

\begin{proof}
For any $\rho \in \Irr_{p'}(H)$, the induced character $\rho^G$ has $p'$-degree, 
so it contains a constituent $\chi \in \Irr_{p'}(G)$. By assumption, $\rho = \chi^*$. Thus $*$ is 
surjective, and so it is injective as well.
\end{proof}

The following result is well known, see for example \cite[22.4]{James}. 

\begin{lem}\label{binom1}
Let $a,b\in\Z_{\geq 0}$, $n=a+b$, and consider $2$-adic decompositions
$$n = \sum^t_{i=1}2^in_i,~~a = \sum^t_{i=1}2^ia_i,b = \sum^t_{i=1}2^ib_i.$$
Then the following statements are equivalent:
\begin{enumerate}[\rm(a)]
\item The binomial coefficient $n \choose a$ is odd.
\item $a_i+b_i = n_i$ for all $i$.
\item $0 \leq a_i \leq n_i$ for all $i$.
\end{enumerate}
\end{lem}

For $r\in\Z_{>0}$ we denote by $[r]_p$ the largest $p$-power that divides $r$; we also set
$[0]_p := \infty$.

\begin{cor}\label{binom2}
Let $a_1, \ldots, a_m\in\Z_{> 0}$ and $n = \sum^m_{i=1}a_i$. Suppose that $n!/\prod^m_{i=1}a_i!$ is odd. Then, by relabeling $a_1, \ldots,a_m$  suitably, we may assume that 
$$[n]_2 = [a_1]_2 < [a_2]_2 < \ldots < [a_m]_2.$$
\end{cor}

\begin{proof}
There is nothing to prove for $m = 1$. Assume $m = 2$, and consider the $2$-adic decompositions 
$$n = \sum^t_{i=k}2^in_i,~~a_1 = \sum^t_{i=0}2^ib_i,~~a_2 = \sum^t_{i=0}2^ic_i,$$ 
with $k \geq 0$ and $n_k=1$. By Lemma \ref{binom1} we have $b_i = c_i = 0$ for $0 \leq i < k$ and, moreover, 
relabeling $a_1$ and $a_2$ if necessary, we may assume that $(b_k,c_k) = (1,0)$. Thus 
$[n]_2 = [a_1]_2 < [a_2]_2$ as needed. Now the case $m > 2$ follows by an easy induction on $m$ using the case $m = 2$.
\end{proof}

\begin{cor}\label{binom3}
Suppose that $a,b\in\Z_{>0}$, $n=a+b$, and ${n \choose a}$ is odd. Then there is a unique $c\in\{a-1,a\}$ such that ${n-1 \choose c}$ is odd. Moreover, if we assume additionally that $[a]_2 \leq [b]_2$, then ${n-1 \choose a-1}$ is odd. 
\end{cor}

\begin{proof}
As ${n \choose a} = {n-1 \choose a-1} +{n-1 \choose a}$, the first claim follows.  For the second claim, let $c\in\{a-1,a\}$ be such that ${n-1 \choose c}$ is odd. By Corollary \ref{binom2}, the assumption $[a]_2 \leq [b]_2$ implies that 
$[n]_2=[a]_2 < [b]_2$. If $n$ is odd, then $a$ is odd and $b$ is even, and, by Lemma \ref{binom1}, $c$ is even, whence $c=a-1$.  
If $n$ is even, we consider the $2$-adic decompositions 
$$n = \sum^t_{i=k}2^in_i,~a = \sum^t_{i=0}2^ia_i,~b = \sum^t_{i=0}2^ib_i,$$ 
with $k \geq 1$ and $n_k = 1$.  By Lemma \ref{binom1}, 
$$(a_0, \ldots ,a_k) = (0, \ldots ,0,1),~(b_0, \ldots ,b_k) = (0, \ldots ,0,0),$$
and so ${n-1 \choose a}$ is even. Hence again we must have that $c= a-1$. 
\end{proof}

Recall that complex irreducible characters of $\SSS_n$ are labelled by partitions 
$\lam \vdash n$: $\chi = \chi^\lam$. By Theorem \ref{sn}, there is a canonical map
$\lam \mapsto \lam^*$ such that, if $\chi^\lam \in \Irr(\SSS_n)$ is of odd degree then 
$\chi^* = \chi^{\lambda^*}$ is the unique odd-degree irreducible constituent of $\chi_{\SSS_{n-1}}$. 

From now on, we fix an odd prime power $q$. For any $n \geq 1$, let $G = GL_n(q)$, with a natural 
module $V = \FF_q^n = \langle e_1, \ldots ,e_n \rangle_{\FF_q}$. As in \cite{KT2}, it is convenient for us to use 
the Dipper-James classification of complex irreducible characters of $G$, as described in \cite{JamesGL}. Namely, every $\chi\in \Irr(G)$ can be written uniquely, up to a permutation of the pairs 
$\{(s_1,\lam_1),\dots,(s_m,\lam_m)\}$, in the form 
\begin{equation}\label{gl1}
  \chi = S(s_1,\lam_1) \circ S(s_2,\lam_2) \circ \ldots \circ S(s_m,\lam_m).
\end{equation}
Here, $s_i \in \bFq^\times$ has degree $d_i$ over $\FF_q$, $\lam_i \vdash k_i$, 
$\sum^m_{i=1}k_id_i = n$, and the $m$ elements $s_i$ have pairwise distinct minimal polynomials over $\FF_q$. 
In particular, $S(s_i,\lam_i)$ is an irreducible character of $GL_{k_id_i}(q)$. 

Furthermore, there is a parabolic subgroup 
$P_\chi = U_\chi L_\chi$ of $G$ with Levi subgroup $L_\chi = GL_{k_1d_1}(q) \times \ldots GL_{k_md_m}(q)$ and 
unipotent radical $U_\chi$. The (outer) tensor product 
$$\psi := S(s_1,\lam_1) \otimes S(s_2,\lam_2) \otimes \ldots \otimes S(s_m,\lam_m)$$
is an $L_\chi$-character, and $\chi$ is obtained from $\psi$ via the Harish-Chandra induction $R^{G}_{L_\chi}$, 
i.e. we first inflate $\psi$ to a $P_\chi$-character and then induce it to $G$. 
The adjoint operation of Harish-Chandra restriction  $\SR^{G}_{L_\chi}$ 
takes any character $\rho$ of $G$, afforded by a $\CC G$-module $W$ to the $L_\chi$-character afforded by 
$W^{U_\chi}$, the fixed point subspace for $U_\chi$ on $W$. 

Let $P=UL$ be a maximal parabolic subgroup of $G$ with Levi subgroup $L = GL_1(q) \times GL_{n-1}(q)$ and 
the unipotent radical $U$. Conjugating suitably in 
$G$ and applying the transpose-inverse automorphism if necessary, we may assume that 
$P = \operatorname{Stab}_{G}(\langle e_1 \rangle_{\FF_q})$ and the second factor $GL_{n-1}(q)$ of $L$ fixes 
both $e_1$ and $\langle e_2, \ldots ,e_n \rangle_{\FF_q}$. 

\smallskip
Given the above notation, we can now prove the following theorem which implies Theorem A:

\begin{thm}\label{main-gl}
Let $q$ be an odd prime power, $n \geq 2$, $G  = GL_n(q)$, $P = \operatorname{Stab}_G(\langle e_1 \rangle_{\FF_q})$. 
Suppose that $\chi \in \Irr_{2'}(G)$. Then the following statements hold:

\begin{enumerate}[\rm(i)]
\item One can choose a label \eqref{gl1} for $\chi$ such that $s_i \in \FF_q^\times$ (so that $d_i = 1$) and 
$\chi^{\lam_i} \in \Irr_{2'}(\SSS_{k_i})$ for all $i=1,\dots,m$, and 
$$[n]_2 = [k_1]_2 < [k_2]_2 < \ldots < [k_m]_2;$$

\item $\chi_P = \chi^* + \Delta$, where $\chi^* \in \Irr_{2'}(P)$ and 
either $\Delta = 0$ or $\Delta$ is a $P$-character all irreducible constituents of which are of even degree;

\item $\chi^*$ is trivial on $U$, and equal to 
$S(s_1,(1)) \otimes (S(s_1,\lam^*_1) \circ S(s_2,\lam_2) \circ \ldots \circ S(s_m,\lam_m))$ 
when viewed as a character of $GL_1(q) \times GL_{n-1}(q)$;

\item If $n$ is odd, then the map $\chi \mapsto \chi^*$ is a bijection between $\Irr_{2'}(G)$ and $\Irr_{2'}(P)$.
\end{enumerate}
\end{thm}

Note that in \ref{main-gl}(iii), the symbol $S(s_1,\lam^*_1)$ is considered void if $k_1 = 1$. We proceed in a series of lemmas.

\begin{lem}\label{order}
Statement {\rm (i)} of Theorem \ref{main-gl} holds.
\end{lem}

\begin{proof}
Since the degree of $\chi$ is odd, so are the degrees of each $S(s_i,\lam_i)$, which implies that 
each $d_i = \deg(s_i) = 1$, for example by \cite[Lemma 5.7(ii)]{KT2}, i.e. $s_i \in \FF_q^\times$ for all $i=1,\dots,m$. Next, we also must have that $|G:P_\chi|$ is odd, which implies by 
a repeated application of \cite[Lemma 4.4(i)]{NT1} that $n!/\prod^m_{i=1}k_i!$ is odd. So we may assume by 
Corollary \ref{binom2} that $[n]_2 = [k_1]_2 < [k_2]_2 < \ldots < [k_m]_2$. Finally, it is well known (and follows from 
the hook formula for the degree of unipotent characters of $G$, see \cite[(1.15)]{FS}) that 
\begin{equation}\label{gl2}
  \chi^\lam(1) \equiv \deg(S(s,\lam)) (\mod 2)
\end{equation}  
if $s \in \FF_q^\times$, and so we conclude that $\chi^{\lam_i} \in \Irr_{2'}(\SSS_{k_i})$.
\end{proof}

\begin{lem}\label{hc}
Let $X = X_1 \times X_2$, where $X_1 \cong GL_m(q)$, $X_2 \cong GL_n(q)$, $P_1 = UL_1$ a parabolic subgroup 
of $X_1$ with unipotent radical $U$ and Levi subgroup $L_1$,  and let $L = L_1 \times X_2$.
\begin{enumerate}[\rm(i)]
\item If $\al$ is a character of $X_1$ and $\beta$ is a character of $X_2$, then 
$\SR^X_L(\al \otimes \beta) = \SR^{X_1}_{L_1}(\al) \otimes \beta$.
\item If $\gam$ is a character of $L_1$ and $\delta$ is a character of $X_2$, then 
$R^X_L(\gam \otimes \delta) = R^{X_1}_{L_1}(\gam) \otimes \delta$.
\end{enumerate}
\end{lem}

\begin{proof}
(i) Let $\al$, respectively $\beta$, be afforded by a $\CC X_1$-module $A$, respectively a $\CC X_2$-module $B$. Then 
$\SR^X_L(\al \otimes \beta)$ is afforded by the $L$-module $(A \otimes B)^U = A^U \otimes B$ and so equal to 
$\SR^{X_1}_{L_1}(\al) \otimes \beta$.

\smallskip
(ii) Inflate $\gam$ to the character $\tilde\gam$ of $P_1$ using $P_1/U \cong L_1$, and inflate $\gam \otimes \delta$ 
to the character $\rho = \tilde\gam \otimes \delta$ of $P = P_1 \times X_2$ using $P/U \cong L$. Then
$$R^X_L(\gam \otimes \delta) = \rho^X = \left((\tilde\gam \otimes 1_{X_2}) \cdot (1_{X_1} \otimes \delta)_P\right)^X=$$
$$= (\tilde\gam \otimes 1_{X_2})^X \cdot (1_{X_1} \otimes \delta) 
    = (R^{X_1}_{L_1}(\gam) \otimes 1_{X_2}) \cdot (1_{X_1} \otimes \delta) = R^{X_1}_{L_1}(\gam) \otimes \delta.$$
\end{proof}

\begin{pro}\label{restr}
Statements {\rm (ii)} and {\rm (iii)} of Theorem \ref{main-gl} hold.
\end{pro}

\begin{proof}
(a) First we note that, since $L$ acts transitively on the $q^{n-1}-1$ non-principal irreducible characters of $U$ and $q$ is odd, 
any irreducible character of $P$ which is nontrivial on $U$ has even degree. Hence all the odd-degree irreducible constituents
of $\chi_P$ are contained in $\SR^G_L(\chi)$. Next, by Lemma \ref{order},  we already know that 
$s_i \in \FF_q^\times$ for all $i=1,\dots,m$. 

Suppose that $m = 1$. Then $\chi$ is a unipotent character of $G$  tensored with a linear character; in particular, 
$\chi$ belongs to the principal series. By the Comparison Theorem \cite[Theorem 5.9]{HL} (see also 
\cite[Theorem 5.1]{C} for the case of the principal series), the computation of $\SR^G_L(\chi)$ can be replaced by the 
computation of $(\chi^{\lam_1})_{\SSS_{n-1}}$, where we identify $\SSS_n$, respectively $\SSS_{n-1}$, with the Weyl group of
$\GC$, respectively of $\LC$. (See also \cite[Proposition (1C)]{FS} for the explicit formula in the case of $G$.)
Applying Theorem \ref{sn} and \eqref{gl2}, we are done in this case.

\smallskip
(b) Now we will assume $m \geq 2$ and set $a= k_1$, $b = n-k_1$, where $k_1, \ldots ,k_m$ satisfy \ref{main-gl}(i); in particular,
\begin{equation}\label{ab}
  [a]_2 < [b]_2,
\end{equation}   
and so $a \neq b$. Let $M=GL_a(q)\times GL_b(q)$ so that $\chi = R^G_M(\alpha \otimes \beta)$, where 
$$\alpha = S(s_1,\lam_1),\quad\beta = S(s_2,\lam_2) \circ \ldots \circ S(s_m,\lam_m).$$
This follows by the transitivity of the Harish-Chandra induction \cite[Proposition 4.7]{DM}. 
By the Mackey formula for Harish-Chandra induction and restriction (see e.g. \cite[Theorem 1.14]{DF}),
\begin{equation}\label{gl3}
\begin{split}
\SR^G_L\chi=  \SR^G_L ( R^G_M(\alpha\otimes \beta) )
=
  &R^L_{L \cap M} \left( \SR^{M}_{L \cap M}(\al\otimes \beta)\right) \\
  &\oplus 
   R^L_{L \cap wMw^{-1}} \left(\operatorname{conj}_w( \SR^{M}_{M\cap w^{-1}L w} (\al\otimes\beta))\right),
   \end{split}
\end{equation}  
where $w$ is the permutation matrix corresponding to the cycle $(1,2,\dots,a+1)$ and $\operatorname{conj}_w$
denotes the conjugation by $w$.

\smallskip
(c)  Here we consider the first summand in the right hand side of \eqref{gl3}.  By Lemma \ref{hc}(i),
$$\SR^M_{L \cap M}(\alpha \otimes \beta) = \SR^{GL_a(q)}_{T_1 \times GL_{a-1}(q)}(\al) \otimes \beta.$$
It is well known (see e.g. \cite[p. 70]{L}) that the Harish-Chandra induction and restriction respect the Lusztig series,
which in our case is labeled by the semisimple element in the dual group $G^* \cong G$ that has each 
$s_i \in \FF_q^\times$ as eigenvalue with multiplicity $k_i$. Hence we can write
\begin{equation}\label{gl4}
  \SR^{GL_a(q)}_{T_1 \times GL_{a-1}(q)}(\al) = \sum^r_{j=1}S(s_1,(1)) \otimes S(s_1,\mu_j)
\end{equation}  
for some $r \geq 1$ and some partitions $\mu_j \vdash (a-1)$. It then follows by Lemma \ref{hc}(ii) and the 
transitivity of the Harish-Chandra induction that
$$R^L_{L \cap M} ( \SR^M_{L \cap M}(\alpha \otimes \beta)) = 
    R^L_{T_1 \times GL_{a-1}(q) \times GL_b(q)}\left(\left(\sum^r_{j=1}S(s_1,(1)) \otimes S(s_1,\mu_j)\right) \otimes \beta\right)$$
$$= \sum^r_{j=1}S(s_1,(1)) \otimes R^{GL_{n-1}(q)}_{GL_{a-1}(q) \times GL_{b}(q)}(S(s_1,\mu_j) \otimes \beta)= 
    \sum^r_{j=1}S(s_1,(1)) \otimes \gam_j$$
where $\gam_j :=  S(s_1,\mu_j) \circ S(s_2,\lam_2) \circ \ldots \circ S(s_m,\lam_m) \in \Irr(GL_{n-1}(q))$. 
Certainly, the 
irreducible constituent $S(s_1,(1)) \otimes \gam_j$ can be of odd degree only when $\deg(S(s_1,\mu_j))$ and 
$|GL_{n-1}(q):(GL_{a-1}(q) \times GL_b(q))|$ are both odd. The former implies, by applying (a) to \eqref{gl4} that  
$\mu_j = \lam^*_1$. The latter implies by \cite[Lemma 4.4(i)]{NT1} that ${n-1 \choose a-1}$ is odd.   

We have shown that the first summand in \eqref{gl3} contains at most one irreducible constituent 
of odd degree, namely
$$S(s_1,(1)) \otimes (S(s_1,\lam^*_1) \circ S(s_2,\lam_2) \circ \ldots \circ S(s_m,\lam_m)),$$ 
which can occur with multiplicity $\leq 1$ in the summand only when ${n-1 \choose a-1}$ is odd.

\smallskip
(d) Arguing as in (c), we can show that each irreducible constituent in the second summand of \eqref{gl3} is of form 
$$S(s'_1,(1)) \otimes (S(s'_2,\nu_1) \circ S(s'_3,\nu_2) \circ \ldots \circ S(s'_m,\nu_{m-1}) \circ S(s_1,\lam_1)),$$
where $s'_1 \in \{s'_2, \ldots ,s'_m\} = \{s_2, \ldots ,s_m\}$ and $\nu_1, \ldots ,\nu_{m-1}$ are some partitions with the lengths totalling to
$b-1$. By \cite[Lemma 4.4(i)]{NT1}, such an irreducible constituent can have odd degree only when ${n-1 \choose b-1}={n-1 \choose a}$ is odd.

Given the condition \eqref{ab}, Corollary \ref{binom3} implies that ${n-1 \choose a}$ is even. Thus the second summand 
in \eqref{gl3} contains no odd-degree irreducible constituent. Since $\chi(1)$ is odd, we are done by the result of (c).   
\end{proof}

\begin{proof}[Completion of the proof of Theorem \ref{main-gl}]
In view of Lemma \ref{order}, Proposition \ref{restr}, and Lemma \ref{trivial}, it remains to prove that 
$|\Irr_{2'}(G)| = |\Irr_{2'}(P)|$ when $n = 2k+1$. We can choose a Sylow $2$-subgroup $S = S_1 \times S_2 < L$ of $G$, where 
$S_1 = \OO_2(GL_1(q))$, and $S_2 \in \Syl_2(GL_{2k}(q))$ has the form $S_2 = A \wr B$, where $A \in \Syl_2(GL_2(q))$ and 
$B \in \Syl_2(\SSS_k)$. Then the $S$-module $V$ decomposes as $V_0 \oplus V'$, where $V_0 = \langle e_1 \rangle_{\FF_q}$
and all irreducible constituents of $V' = \langle e_2, \ldots ,e_n \rangle_{\FF_q}$ are of even dimension. It follows that 
$\NB_G(S)$ is contained in $\operatorname{Stab}_G(V_0) \cap \operatorname{Stab}_G(V') = L$ and so $\NB_G(S) = \NB_L(S)$. By the main result of \cite{O1}, 
\begin{equation}\label{gl5}
  |\Irr_{2'}(G)| = |\Irr_{2'}(\NB_G(S))| = |\Irr_{2'}(\NB_L(S))| = |\Irr_{2'}(L)|.
\end{equation}  
As mentioned in part (a) of the proof of Proposition \ref{restr}, $|\Irr_{2'}(L)| = |\Irr_{2'}(P)|$, and so we are done.
\end{proof}

\begin{proof}[Proof of Theorem B]
Let $G := GL_n(q)$, and consider normal subgroups $S := SL_n(q)$ and $H := \ZB(G)S$ of $G$.

\smallskip
(i) Consider any $\theta \in \Irr_{2'}(S)$. Since $H = \ZB(G) * S$ is a central product, $\theta$ extends to $H$ which has odd index 
$\gcd(n,q-1)$ in $G$. Hence $\theta$ lies under some $\chi \in \Irr_{2'}(G)$. Now we can find a label \eqref{gl1} for $\chi$ that 
satisfies \ref{main-gl}(i). Applying \cite[Theorem 1.1]{KT1} to $\chi$, we then see that $\chi_S$ is irreducible, i.e. $\chi_S = \theta$. 
We have shown that every $\theta \in \Irr_{2'}(S)$ extends to $G$, when
\begin{equation}\label{sl1}
  |\Irr_{2'}(S)| = |\Irr_{2'}(G)|/(q-1)
\end{equation}
as $G/S \cong C_{q-1}$.  

\smallskip
(ii) We may assume that $Q = P \cap S$, where $P$ is the maximal parabolic subgroup of $G$ considered in Theorem \ref{main-gl}.
Then $Q = UL_1$ with $L_1:= L \cap S \cong GL_{n-1}(q)$. Note that the normal subgroup
$$K := \operatorname{Stab}_S(e_1) \cap \operatorname{Stab}_S(\langle e_2, \ldots ,e_n\rangle_{\FF_q}) \cong SL_{n-1}(q)$$ 
of $L_1$ acts transitively on $\Irr(U) \smallsetminus \{1_U\}$ since $n \geq 3$. It follows that every irreducible character of
$Q$ that is nontrivial on $U$ has even degree.  In particular,
\begin{equation}\label{sl2}
  |\Irr_{2'}(Q)| = |\Irr_{2'}(L_1)|.
\end{equation}

\smallskip
(iii) Now we keep the notation of (i) and 
consider any odd-degree irreducible constituent $\zeta$ of $\theta_Q$ (which exists since $\theta$ has odd degree). 
We have shown that $\zeta$ is trivial on $U$ and so is a constituent of $\SR^G_L(\chi)$. We can be view $\zeta$ as an $L_1$-character. 
Any irreducible character of $\ZB(G)L_1 = \ZB(G) * L_1$ lying above $\zeta$ then also has degree equal to $\zeta(1)$. 
Next, $\ZB(G)L_1$ has odd index $\gcd(n,q-1)$ in $L$. Thus any irreducible character of $L$ lying above $\zeta$ must have odd degree. 
Applying Theorem \ref{main-gl}, we now see that $\zeta$ is a constituent of $(\chi^*)_{L_1}$, and 
$$(\chi^*)_{\tilde K} = S(s_1,\lam^*_1) \circ S(s_2,\lam_2) \circ \ldots \circ S(s_m,\lam_m)$$
for 
$$\tilde K :=  \operatorname{Stab}_G(e_1) \cap \operatorname{Stab}_G(\langle e_2, \ldots ,e_n\rangle_{\FF_q}) \cong GL_{n-1}(q).$$ 
Since $\chi^*(1)$ is odd,  
by \cite[Lemma 4.4(i)]{NT1} and Corollary \ref{binom2} we have that $[k_1-1]_2$, $[k_2]_2$, $\ldots$, $[k_m]_2$ are pairwise different;
in particular, $k_1-1, k_2, \ldots ,k_m$ are pairwise different. It then follows by \cite[Theorem 1.1]{KT1} that $(\chi^*)_K$ is irreducible. 
Hence, $(\chi^*)_{L_1}$ is irreducible and $\zeta = (\chi^*)_{L_1}$.

\smallskip
(iv) To complete the proof of Theorem B, by Lemma \ref{trivial} it remains to show that $|\Irr_{2'}(S)| = |\Irr_{2'}(Q)|$. 
Since 
$$|\Irr_{2'}(L)| = |\Irr_{2'}(GL_1(q) \times GL_{n-1}(q))| = (q-1)|\Irr_{2'}(GL_{n-1}(q))| = (q-1)|\Irr_{2'}(L_1)|,$$
we are done by combining \eqref{gl5}, \eqref{sl1}, and \eqref{sl2}.
\end{proof}

\section{Finite groups with self-normalizing Sylow subgroups}

In this section we prove Theorem C. In fact, in this case the key hypothesis
is the self-normalizing Sylow and not the nature of the prime $p$.
The next result was proved in \cite{N1} for $P=H$ (with a more complicated proof).

\begin{thm}\label{solvable}
Suppose that $G$ is a $p$-solvable group, $P \in \syl pG$,
and assume that $P=\norm GP$. Let $P \leq H \le G$.
If $\chi \in \irr G$ has $p'$-degree, then the restriction
of $\chi$ to $H$ contains a unique $p'$-degree irreducible constituent $\psi \in \irr H$.
This establishes a natural bijection between the $p'$-degree irreducible
characters of $G$ and $H$.
\end{thm}

\begin{proof} We proceed by induction on $|G|$. Let $\chi \in \irr G$ be
of $p'$-degree.
Let $N$ be a normal subgroup of $p'$-order.
Since $\chi(1)$ is not divisible by $p$, using the Clifford correspondence (Theorem (6.11) of \cite{Is}),
let $\theta \in \irr N$ be $P$-invariant under $\chi$. Then $\cent NP=1$ by hypothesis,
and $\theta=1_N$, by the Glauberman correspondence
(Theorem (13.1) of \cite{Is}). Therefore,
$N\leq \ker\chi$.  Now consider $\bar\chi \in \irr{G/N}$ the corresponding
character in the factor group.
By induction $\bar\chi_{NH}$ contains a unique $p'$-irreducible constituent
$\psi$. Now, all the constituents of this character restrict irreducibly to $H$,
so we are done in this case. So we may assume that
$N=1$. Let now $N$ be a normal $p$-subgroup of $G$.
Then $\chi_P$ contains a linear constituent $\lambda$ and
$\lambda_N=\nu \in \irr N$. Let $P \leq T$ be the stabilizer
of $\nu$ in $G$, and let $\psi \in \irr T$ be the Clifford correspondent of $\chi$
over $\nu$. It is well-known by Clifford theory (see \cite[Lemma 2.5]{S}) that
$$\chi_H=(\psi_{T\cap H})^H + \Delta \, $$
where no irreducible constituent of $\Delta$ lies over $\nu$.
We claim that no irreducible constituent of $\Delta$
has $p'$-degree. If $\xi$ is an irreducible constituent
of $\Delta$ of $p'$-degree, then $\xi_N$ has some linear
$P$-invariant constituent $\epsilon$. But then $\chi_N$ has
$P$-invariant irreducible constituents $\nu$ and $\epsilon$, 
and by a standard argument they are $\norm GP$-conjugate.
However, $\norm GP=P$, and since both are $P$-invariant, they
are equal. This contradicts the nature of $\Delta$.
Suppose now that $T<G$.  Notice that $T$ has a self-normalizing
Sylow $p$-subgroup. By induction,
$$\psi_{T\cap H}=\mu  + \rho\,,$$
where $\mu$ has $p'$-degree and every irreducible
constituent of $\rho$ has degree divisible by $p$.
All these constituents lie over $\nu$, so all of them induce
irreducibly to $H$ by the Clifford correspondence.
We conclude that
$$\chi_H=(\psi_{T\cap H})^H + \Delta=\mu^H + \rho^H + \Delta \, ,$$
and the theorem is again proved in this case.

Hence the last case is that $T=G$. We claim that $\nu$ extends to
$G$. This is because $\nu$ extends to $P$,  by coprimeness it
extends to every   $Q/N\in \syl q{G/N}$ if $q \ne p$ (by Corollary (6.27) of \cite{Is}),
and therefore it extends to $G$ by Theorem (6.26)
of \cite{Is}. Let $\lambda\in \irr G$
be a linear extension of $\nu$.  By Gallagher's Corollary (6.17) of \cite{Is},
we can write $\chi=\beta\lambda$, where $\beta \in \irr{G/N}$,
has $p'$-degree.  By induction,
$$\beta_H=\tau + \Xi\,,$$ 
where every irreducible constituent of $\Xi$ has degree divisible by $p$.
Now
$$\chi_H=\lambda_H \beta_H =\lambda_H\tau + \lambda_H\Xi \, ,$$
and this proves half of the theorem. The bijectivity of the map is proved
similarly.
\end{proof}

The fact that $\norm GP=P$ is key. For instance,
take the product of an extraspecial group $F = 3^{1+2}_+$ of 
exponent $3$ acted on by $C_4$ in such a way that $C_4$ acts trivially
on the center of $F$. The normalizer of a Sylow $2$-subgroup
in $G$ is $N=C_4 \times C_3$, is even maximal.
Every restriction of the irreducible characters of degree 3 has
three irreducible linear constituents.

\begin{proof}[Proof of Theorem C]
In view of Theorem \ref{solvable}, we may assume $p > 2$. 
Decompose $\chi_H =\sum^t_{i=1}\rho_i + \sum^s_{j=1}\psi_j$,
where $t \geq 1$, $s \geq 0$, $\rho_i,\psi_j \in \Irr(H)$, $p \nmid \rho_i(1)$, $p |\psi_j(1)$.
Now each $\rho_i|_P$ contains a linear character, and by \cite[Corollary B]{NTV},
$\chi_P = \lambda + \delta$ contains a unique linear character
$\lambda$ (so that either $\delta = 0$ or all irreducible constituents of $\delta$ have degree 
divisible by $p$). Hence $t = 1$, and
we have a well-defined map
$$*: \chi \mapsto \chi^* := \rho_1$$ 
from $X := \Irr_{p'}(G)$ to $Y := \Irr_{p'}(H)$.

Now, for any $\rho \in Y$, $\rho^G$ contains a $p'$-degree $\chi$ of $G$ and
then $\chi_H$ contains $\rho$, so $\chi^* = \rho$. Thus $*$ is onto. By \cite[Corollary B]{NTV}, 
$|X| = |P/P'| = |Y|$. So $*$ is a bijection.
\end{proof}

\section{Canonical character correspondences in symmetric groups}\label{SD}
Note that the natural character correspondences described in \cite{APS} (see Theorem \ref{sn}) and 
in Theorems A and B are given by restriction. In the case of $G = \SSS_n$, one may ask whether 
the restriction to a Young subgroup $H$ of odd index in $G$ would also give such a correspondence
between the odd-degree characters of $G$ and those of $H$. Unfortunately, the answer is no: 
if $G=\SSS_7$ and $H=\SSS_5 \times \SSS_2$, then the restriction to $H$ of any irreducible character
of degree $35$ of $G$ have three odd-degree irreducible constituents.

\smallskip
When $G = \SSS_n$ and $P \in \Syl_2(G)$, a canonical bijection between $\Irr_{2'}(G)$ and $\Irr_{2'}(P)$
was defined in \cite{G} (although given by restriction to $P$ only when $n$ is a $2$-power). 
We will construct a new such canonical bijection, where, in addition, the correspondent $\chi^\sharp \in \Irr_{2'}(P)$ of 
$\chi \in \Irr_{2'}(G)$ is an irreducible constituent of $\chi_P$ (although not necessarily occurring with multiplicity one).
Let $\HC(n)$ denote the set of $n$ hook partitions $(n-i,1^i)$, $0 \leq i < n$, of $n$.  
If $\lam \vdash n$, then $Y(\lam)$ is the Young diagram of $\lam$. Furthermore,
we will say that $\lam$ is an {\it odd} partition of $n$ (and write $\lam \vdash_\orm n$) exactly 
when $\chi^\lam \in \Irr(\SSS_n)$ has odd degree.

\begin{lem}\label{hook}
Let $n \in \Z_{\geq 1}$ and let $\gam$ be a partition of $n$ with a hook $H$ of length $m$ in $Y(\gam)$. Let $\beta \in \HC(m)$ be the 
hook partition corresponding to $H$, and let $\al \vdash (n-m)$ be such that $Y(\al)$ is obtained from $Y(\gam)$ by removing the 
rim $m$-hook $R$ of $Y(\gam)$ corresponding to $H$. Then $(\chi^\gam)_{\SSS_{n-m} \times \SSS_m}$ contains $\chi^\al \otimes \chi^\beta$ 
as an irreducible constituent with multiplicity one.
\end{lem}

\begin{proof}
We apply the Littlewood-Richardson rule as given in \cite[Corollary 2.8.14]{JK}. Write $\beta = (k,1^{m-k})$ for some $1 \leq k \leq m$ and 
consider the symbols $a_{ij}$ where the $(i,j)$-node belongs to $Y(\beta)$:
$$a_{11}, a_{12}, \ldots ,a_{1k},a_{21}, a_{31}, \ldots, a_{m-k+1,1}.$$
We need to count the number of ways of adding symbols $a_{ij}$ to $Y(\al)$ to get 
$Y(\gam)$, that is, to fill up the rim $m$-hook $R$ of $Y(\gam)$ with these symbols $a_{ij}$ in such a way that all conditions (i), (ii), (iii)
of \cite[Corollary 2.8.14]{JK} are fulfilled. First we have to put the $k$ symbols $a_{11}, a_{12}, \ldots ,a_{1k}$ in $k$ different columns of
$R$ consecutively starting from the rightmost to get a proper diagram. Since $R$ has exactly $k$ columns, we have to put these symbols at the top
of these $k$ columns, and there is exactly one way to do it. Then we need to put the $m-k$ symbols $a_{21}, a_{31}, \ldots,a_{m-k+1,1}$
in the $m-k$ remaining rows of $R$ consecutively starting from the highest row. Since there remain only $m-k$ rows, each with one node, 
there is again exactly one way to do it. 
\end{proof}

We refer to \cite[\S2.7]{JK} for the notion of the $m$-core of any partition $\lam \vdash n$ and any $m \in \Z_{\geq 1}$.

\begin{lem}\label{core}
Let $m \in \Z_{\geq 2}$, $n \in \Z$ be such that $m \leq n \leq 2m-1$, and let 
$\al = (a_1,a_2, \ldots ,a_r) \vdash (n-m)$ with $a_1 \geq a_2 \geq \ldots \geq a_r > 0$.

\begin{enumerate}[\rm(i)]
\item For any $\beta = (k,1^{m-k}) \in \HC(m)$, there is exactly one $\gam \vdash n$ such that 
$\gam$ has a hook $H$ of length $m$ that corresponds to $\beta$, and, furthermore, $Y(\al)$ is obtained
from $Y(\gam)$ by removing the rim $m$-hook corresponding to $H$.

\item Suppose in addition that $m$ is a $2$-power and that $\al \vdash_\orm (n-m)$. Then, for each
$\beta = (k,1^{m-k}) \in \HC(m)$, there is a unique $\lam \vdash_\orm n$ such that $\gam$ has hook
$H$ of length $m$ that corresponds to $\beta$, and, furthermore, $\al$ is the $m$-core of $\gam$.   
\end{enumerate}
\end{lem}

\begin{proof}
(i) We determine in how many ways $Y(\al)$ can be obtained from $Y(\gam)$ by removing a rim $m$-hook $R$ that corresponds to $H$,
a hook of length $m$ with associated partition $\beta \in \HC(m)$; in particular, $R$ spans $m-k+1$ rows and $k$ columns.  
Since $n \leq 2m-1$, $H$ must be the hook corresponding to a node that belongs either to the first row or the first column of $Y(\gam)$. Thus 
$R$ must touch either the first row or the first column of $Y(\al)$. Consider the outer rim $S$ of $Y(\al)$, which spans 
$r+1$ rows and $a_1+1$ columns and so has length $L:= a_1+r+1$. Note that
$$a_1+r \leq a_1 + a_2 + \ldots +a_r + 1 = n-m+1 \leq m.$$

Suppose first that $m-k+1 \leq r$. Then $R$ spans $k \geq m-r+1 \geq a_1+1$ columns.  
If $R$ does not touch the first row of $Y(\al)$, then it must touch the first column of $Y(\al)$ and 
spans at most $a_1$ columns, a contradiction. Thus $R$ must begin in the first row of $Y(\al)$. In this case, since 
$k \leq r$, $R$ consists precisely of the $N_1$ nodes that belong to the first $m-k+1$ rows of $S$, together with $m-N_1$ more nodes
in the first row, to the right of the top node of $S$. (Note that, since the $(r+1)$th row of $S$ has at least two nodes, we have
$N_1 \leq L-2 \leq m-1$.) Thus $\gam$ exists and is unique in this case.

Next suppose that $m-k+1 \geq r+1$. If $R$ does not touch the first column of $Y(\al)$, then $R$ spans at most $r$ rows, a contradiction.
Hence $R$ must start from the first column of $Y(\al)$. If in addition $k \leq a_1$, then 
$R$ consists precisely of the $N_2$ nodes that belong to the first $k$ columns of $S$, together with $m-N_2$ more nodes
in the first column, below the last row of $S$. (Note that, since the $(a_1+1)$th row of $S$ has at least two nodes, we have
$N_2 \leq L-2 \leq m-1$.) Thus $\gam$ exists and is unique in this case. Finally, if $k \geq a_1+1$, then 
$R$ consists of $S$ together with $k-(a_1+1)$ more nodes in the first row (to the right of the top node of $S$), and 
$(m-k+1)-(r+1)$ additional nodes in the first column (below the last row of $S$). Again, $\gam$ exists and is unique in this case.

\smallskip
(ii) By \cite[Lemma 1]{APS} (see also \cite[\S6]{O2}),
there are exactly $m$ odd partitions $\lam_i \vdash_\orm n$, $1 \leq i \leq m$, such that the $m$-core of $\lam_i$ 
is $\alpha$, and, furthermore, each $\lam_i$ has a unique hook $H_i$ of length $m$. As $n \leq 2m-1$, $Y(\al)$ is obtained from 
$Y(\lam_i)$ by removing the rim $m$-hook corresponding to $H_i$. Each $H_i$ corresponds to exactly one of $m$ 
hook partitions $\beta_i \in \HC(m)$, and, as we showed in (i), $\lam_i$ is uniquely determined by $\beta_i$. It follows that 
$\beta_1, \beta_2, \ldots ,\beta_m$ are pairwise distinct and so are precisely the $m$ hook partitions of $m$.
\end{proof}

\begin{thm}\label{sym}
Let $n \in \Z_{\geq 1}$ and let $n = 2^{n_1} + 2^{n_2} + \ldots + 2^{n_r}$ be the $2$-adic decomposition of $n$, where 
$n_1 > n_2 > \ldots > n_r \geq 0$. Let $G = \SSS_n$ and $P \in \Syl_2(G)$. 
\begin{enumerate}[\rm (i)]
\item There are canonical bijections 
$$\Irr_{2'}(G) \stackrel{\al}{\longrightarrow} \Theta(n) \stackrel{\beta}{\longleftarrow} \Irr_{2'}(P)$$
  with $\Theta(n) := \HC(2^{n_1}) \times \HC(2^{n_2}) \times \ldots \times \HC(2^{n_r})$.
\item If $\al(\chi) = (\mu_1,\mu_2, \ldots ,\mu_r)$ for $\chi \in \Irr_{2'}(G)$, then the restriction of $\chi$ to 
  the Young subgroup $\SSS_{2^{n_1}} \times \SSS_{2^{n_2}} \times \ldots \times \SSS_{2^{n_r}}$ contains 
  $\chi^{\mu_1} \otimes \chi^{\mu_2} \otimes \ldots \otimes \chi^{\mu_r}$ as an irreducible constituent.  
\item The map $\chi \mapsto \chi^{\sharp} := \beta^{-1}(\alpha(\chi))$ is a canonical bijection between $\Irr_{2'}(G)$ and $\Irr_{2'}(P)$.
  Furthermore, $\chi^\sharp$ is an irreducible constituent of $\chi_P$.
\end{enumerate}
\end{thm}

\begin{proof}
We can take $P = P_1 \times P_2 \times \ldots \times P_r$ with $P_i \in \Syl_2(\SSS_{2^{n_i}})$. Then any $\lam \in \Irr_{2'}(P_i)$ is
of the form $\lam = \lam_1 \otimes \lam_2 \otimes \ldots \otimes \lam_r$ with $\lam_i \in \Irr_{2'}(P_i)$. By Lemma 3.1 and Theorem 3.2 of 
\cite{G}, there is a unique $\nu_i \in \HC(2^{n_i})$ such that $\lam_i$ is the unique odd-degree irreducible constituent of 
the restriction of $\chi^{\nu_i} \in \Irr_{2'}(\SSS_{2^{n_i}})$ to $P_i$. Now we can define 
$$\beta(\lam) := (\nu_1, \nu_2, \ldots ,\nu_r).$$ 

Next, consider any $\chi = \chi^\pi \in \Irr_{2'}(\SSS_n)$, so that $\pi \vdash_\orm n$. By \cite[Lemma 1]{APS}, $\pi$ contains a unique
$2^{n_1}$-hook $\mu_1 \in \HC(2^{n_1})$ and $\pi_2 \vdash_\orm (n-2^{n_1})$, where $\pi_2$ is its $2^{n_1}$-core of 
$\pi$. Conversely, any such pair $(\mu_1,\pi_2)$ can be obtained in this way from a unique $\pi \vdash_\orm n$ by Lemma \ref{core}.
As $Y(\pi_2)$ is obtained from $Y(\pi)$ by removing the rim $2^{n_1}$-hook corresponding to $\mu_1$,
by Lemma \ref{hook} we have that the restriction of $\chi^\pi$ to $\SSS_{2^{n_1}} \times \SSS_{n-2^{n_1}}$ contains $\chi^{\mu_1} \otimes \chi^{\pi_2}$ as 
an irreducible constituent.
Applying this process to $\pi_2$ we then get a pair $(\mu_2,\pi_3)$ with $\mu_2 \in \HC(2^{n_2})$ and 
$\pi_3 \vdash_\orm (n-2^{n_1}-2^{n_2})$. Repeating this process successively, we get $\mu_i \in \HC(2^{n_i})$ for all $i$, and can then 
define 
$$\alpha(\chi) := (\mu_1, \mu_2,\ldots,\mu_r).$$ 
One easily checks that both $\alpha$ and $\beta$ are bijections, and that (ii) and (iii) are fulfilled.
\end{proof}

Note that irreducible characters of $\SSS_n$ are rational, and the same is true for odd-degree irreducible characters of $P \in \Syl_2(\SSS_n)$,
see \cite[Lemma 3.3]{NT2}.

\begin{cor}\label{young}
Let $G = \SSS_n$ and let $Y = \SSS_{k_1} \times \SSS_{k_2} \times \ldots \times \SSS_{k_m}$ be a Young subgroup of 
odd index in $G$. Then there is a canonical bijection $\chi \mapsto \chi^\star$ between $\Irr_{2'}(G)$ and $\Irr_{2'}(Y)$.
\end{cor}

\begin{proof}
For each $i$, choose $P_i \in \Syl_2(\SSS_{k_i})$. Then $P := P_1 \times P_2 \times \ldots \times P_m \in \Syl_2(G)$.  
Consider any $\chi \in \Irr_{2'}(G)$ and let $\lam := \chi^\sharp \in \Irr_{2'}(P)$ as in Theorem \ref{sym}.
Next, write $\lam = \lam_1 \otimes \lam_2 \otimes \ldots \otimes \lam_m$ with $\lam_i \in \Irr_{2'}(P_i)$.  For each $i$ there is 
a unique $\psi_i \in \Irr_{2'}(\SSS_{k_i})$ such that 
$(\psi_i)^\sharp = \lam_i$ by Theorem \ref{sym}. Now we can define 
$$\chi^\star = \psi := \psi_1 \otimes \psi_2 \otimes \ldots \otimes \psi_m \in \Irr_{2'}(Y).$$
It is easy to check that the map $\chi \mapsto \chi^\star$ is a bijection between $\Irr_{2'}(G)$ and $\Irr_{2'}(Y)$.
\end{proof}

We are now ready to prove Theorem D.

\begin{proof}[Proof of Theorem D]
(i) Let $M$ be a maximal subgroup of $G = \SSS_n$ of odd index. If $1 \leq n \leq 4$, then $M \in \Syl_2(G)$ and so
we can set $\chi^\star = \chi^\sharp$ using Theorem \ref{sym}.  Thus we may assume that $n \geq 5$. Maximal subgroups of odd
index of $G$ are determined by the main result of \cite{LS} and independently by \cite[Theorem C]{K}, and they are either
maximal Young subgroups or stabilizers $\SSS_k \wr \SSS_t$ of set  partitions of $\{1,2, \ldots ,n=kt\}$ into $t$ subsets of size $k$.
In the former case we can apply Corollary \ref{young} (with $m = 2$). 

So we will assume that $M = \SSS_k \wr \SSS_t$, and write
$M = N \rtimes \SSS_t$ where $N := \SSS_k \times \SSS_k \times \ldots \times \SSS_k \cong (\SSS_m)^t$
and $\SSS_t$ naturally permutes the $t$ direct factors of $N$.  

\smallskip
(ii) Consider any $\varphi \in \Irr_{2'}(M)$. 
Since $\varphi(1)$ is odd, the length of the $\SSS_t$-orbit of any irreducible constituent of $\varphi_N$ is odd.
Hence, by Corollary \ref{binom2} and using the action of $\SSS_t$ on $N$, we can find an irreducible constituent
$$\psi = \underbrace{\psi_1 \otimes \ldots \otimes \psi_1}_{t_1} \otimes  \underbrace{\psi_2 \otimes \ldots \otimes \psi_2}_{t_2}
              \otimes \ldots \otimes  \underbrace{\psi_m \otimes \ldots \otimes \psi_m}_{t_m}$$
of $\varphi_N$, where $m \geq 1$, $1\leq [t_1]_2 < [t_2]_2 < \ldots < [t_m]_2$,
$\sum^m_{i=1}t_i = t$, and $\psi_1, \ldots ,\psi_m \in \Irr_{2'}(\SSS_k)$ are pairwise distinct.
Note that each $\SSS_t$-orbit of odd length on $\Irr_{2'}(N)$ contains 
a unique representative $\psi$ satisfying these conditions. 
Then $Y:= \operatorname{Stab}_{\SSS_t}(\psi) = \SSS_{t_1} \times \SSS_{t_2} \times \ldots \times \SSS_{t_m}$ is a Young subgroup 
of odd index in $\SSS_t$, and the inertia group $J$ of $\psi$ in $M$ is precisely $N \rtimes Y$. Note that $\psi$ has a canonical extension
$\tilde\psi$ to $J$: if $\psi_i$ is afforded by a $\CC \SSS_k$-module $V_i$, then we can let $Y$ act on
$$V := \underbrace{V_1 \otimes \ldots \otimes V_1}_{t_1} \otimes  \underbrace{V_2 \otimes \ldots \otimes V_2}_{t_2}
              \otimes \ldots \otimes  \underbrace{V_m \otimes \ldots \otimes V_m}_{t_m}$$             
via naturally permuting the tensor factors, and then take $\tilde\psi$ to be the $J$-character afforded by $V$. By the Clifford correspondence
and Gallagher's Corollary  \cite[(6.17)]{Is} (or directly by \cite[Theorem 4.4.3]{JK}), we have 
$$\varphi = (\tilde\psi\al)^M,$$
where $\al \in \Irr_{2'}(Y)$. Thus each $\varphi \in \Irr_{2'}(M)$ defines a unique pair $(\psi,\al)$ in a canonical way. 

\smallskip
(iii) Fix $R \in \Syl_2(\SSS_k)$. 
Then we can find 
$Q \in \Syl_2(\SSS_t)$ such that $P = R \wr Q = R^t \rtimes Q$ is a Sylow $2$-subgroup of $M$. 

Consider any $\lam \in \Irr_{2'}(P)$. Then, as in (ii), after a suitable $\SSS_t$-conjugation, we can write 
$$\mu := \lam_{R^t} = \underbrace{\mu_1 \otimes \ldots \otimes \mu_1}_{s_1} \otimes  \underbrace{\mu_2 \otimes \ldots \otimes \mu_2}_{s_2}
              \otimes \ldots \otimes  \underbrace{\mu_{m'} \otimes \ldots \otimes \mu_{m'}}_{s_{m'}},$$
where $m' \geq 1$, $1\leq [s_1]_2 < [s_2]_2 < \ldots < [s_{m'}]_2$,
$\sum^{m'}_{i=1}s_i = t$, and $\mu_1, \ldots ,\mu_{m'} \in \Irr_{2'}(R)$ are pairwise distinct.
Note that each $\SSS_t$-orbit of odd length on $\Irr_{2'}(R^t)$ contains 
a unique representative $\mu$ satisfying these conditions. 
Then $Y':= \operatorname{Stab}_{\SSS_t}(\mu) = \SSS_{s_1} \times \SSS_{s_2} \times \ldots \times \SSS_{s_{m'}}$ is a Young subgroup of 
$\SSS_t$. Since $\mu$ is $Q$-invariant, $Y'$ has odd index in $\SSS_t$, and the inertia group $J'$ of $\mu$ in 
$R^t \rtimes \SSS_t$ is precisely $R^t \rtimes Y'$. We can replace $Q$ by $Q' := Q_1 \times Q_2 \times \ldots \times Q_{m'}$
where $Q_i \in \Syl_2(\SSS_{s_i})$. As in (ii), $\mu$ has a canonical extension $\tilde\mu$ to $J'$. We now have that
$$\lam = (\tilde\mu)_P\cdot\beta,$$
where $\beta \in \Irr_{2'}(Q')$. Thus each $\lam \in \Irr_{2'}(P)$ defines a unique pair $(\mu,\beta)$ in a canonical way. 

\smallskip
(iv) Now we can define $\chi^\star$ for each $\chi \in \Irr_{2'}(G)$ as follows. First, let $\lam := \chi^\sharp \in \Irr_{2'}(P)$ as in 
Theorem \ref{sym}. Then we can apply the analysis in (iii) to $\lam$ to get the canonical pair $(\mu,\beta)$. We will now rename $m'$ by $m$,
$s_i$ by $t_i$, and $Y'$ by $Y$. For each $i$, there is a unique $\psi_i \in \Irr_{2'}(\SSS_k)$ such that $\psi_i^\sharp = \mu_i$. Next,
$$\beta = \beta_1 \otimes \beta_2 \otimes \ldots \otimes \beta_m,$$
where $\beta_i \in \Irr_{2'}(Q_i)$. Again, for each $i$ there is a unique $\al_i \in \Irr_{2'}(\SSS_{t_i})$ such that 
$\al_i^\sharp = \beta_i$, and then we take 
$$\al := \al_1 \otimes \al_2 \otimes \ldots \otimes \al_m \in \Irr_{2'}(Y).$$
Now we define $\chi^\star := \varphi$, where $\varphi$ corresponds to $(\psi,\al)$ as in (ii). It is straightforward to check that 
this map is a bijection between $\Irr_{2'}(G)$ and $\Irr_{2'}(M)$.
\end{proof}
 
\section{Canonical character correspondences in finite general linear and unitary groups}\label{SE} 
For $\kappa = \pm$, we let $GL^\kappa_n(q)$ denote $GL_n(q)$ when $\kappa = +$ and $GU_n(q)$ when $\kappa = -$.
Let $C_\kappa$ denote the unique subgroup of order $q-\kappa 1$ of $\FF_{q^2}^\times$, and let 
$\tCk$ denote the character group of $C_\kappa$ (under the pointwise product). We will fix a generator $\varep$ of $C_\kappa$,  
a primitive $(q-\kappa 1)$th root of unity $\tv \in \CC$. Then, for $s=\varep^j \in C_\kappa  \leq \FF_{q^2}^\times$,  let $\hat{s}$ denote the character that sends 
$\varep$ to $\tv^j$. We can then identify 
\begin{equation}\label{dual}
  \Irr(\OB_{2}(C_\kappa)) \times \Irr(\OB_{2'}(C_\kappa)) \longleftrightarrow  \tCk \longleftrightarrow C_\kappa,
\end{equation}  
where the latter is given by $\hat{s} \mapsto s$. Let $\Gamma$ denote the Galois group $\Gal(\QQ(\tv)/\QQ)$, and let $\Gamma$ act on 
$C_\kappa$ via $\sigma(\varep) = \varep^i$ whenever $\sigma(\tv) = \tv^i$. We will fix a suitable basis of the natural module $V = \FF_q^n$,
resp. $\FF_{q^2}^n$, for $G$, and use it to define the Frobenius automorphism $F_p$ sending any matrix $X := (a_{ij})$ to $(a_{ij}^p)$ in that basis
(where $p$ is the unique prime divisor of $q$) 
and, additionally, the inverse-transpose automorphism $\tau: X \mapsto \tw tX^{-1}$ when $\kappa = +$. Let
\begin{equation}\label{outer}
  D:= \langle \tau,F_p \rangle,
\end{equation}   
and let $D$ act on $C_\kappa$ and correspondingly on $\tCk$ via
$$F_p(s) = s^p,~F_p(\hat{s}) = \hat{s}^p,~\tau(s) = s^{-1},~\tau(\hat{s}) = \bar{\hat{s}} = \hat{s}^{-1}.$$
We also let $D$ act trivially on $\HC(n)$,  the set of hook partitions of $n$.  

For $\lam \vdash n$, let $\varphi^\lam$ denote the unipotent character
of $GL^\kappa_n(q)$ labeled by $\lam$. By \cite[Proposition 13.30]{DM}, $S(s,\lam) = \varphi^\lam(\hat{s} \circ {\mathrm {det}})$ if $\kappa = +$.
If $\kappa = -$, we also define 
$$S(s,\lam) := \varphi^\lam(\hat{s} \circ {\mathrm {det}}).$$ 
 
\begin{lem}\label{power}
Let $n = 2^m$ for some $m \in \Z_{\geq 0}$, $q$ be an odd prime power, $\kappa = \pm$, $G = GL^\kappa_n(q)$, and 
$P \in \Syl_2(G)$. 
\begin{enumerate}[\rm(i)]
\item The field of values of any character in $\Irr_{2'}(\NB_G(P))$ is contained in $\QQ(\tv)$.
\item There is a canonical bijection between $\Irr_{2'}(\NB_G(P))$ and 
$\tCk \times \HC(n)$ that commutes with the action of $\Gamma$. Furthermore, we can choose $P$ to be $D$-invariant and the 
canonical bijection to be $D$-equivariant. 
\item $\Irr_{2'}(\NB_G(P))$ contains exactly $2^{m+1}$ real characters,  and all of 
them are also rational.
\end{enumerate}
\end{lem}

\begin{proof}
It is well known that 
\begin{equation}\label{norm1}
  \NB_G(P) = ZP = Z_1 \times P,
\end{equation}
where $Z:=\ZB(G)$ and 
$Z_1 := \OB_{2'}(Z)$ (see eg. \cite[Theorem 1]{Ko} for $n > 2$). Note that $Z$ can be canonically 
identified with $C_\kappa$. The case $m=0$ is obvious, so we will assume $m \geq 1$.

\smallskip
(a) Consider the case $n = 2$. Suppose $q \equiv \kappa 1 (\mod 4)$. Then we can take a basis, which is orthonormal if $\kappa = -$,
of $V_1 = \FF_q^2$ or $\FF_{q^2}^2$, define $F_p$ and $\tau$ in this basis, and choose 
$P = \NB_G(T_2) = T_2 \rtimes \langle z \rangle$, where $T_2 = \OB_2(T)$,
$T := \{ \diag(x,y) \mid x,y \in C_\kappa\}$, and $|z| = 2$. It is easy to see that each $\lam \in \Irr_{2'}(P)$ is 
uniquely determined by $\gam \in \Irr(\OB_2(C_\kappa))$ and $j \in \{0,1\}$, where $\lam(g) = \gam(xy)$
for $g = \diag(x,y) \in T_2$ and $\lam(z) = (-1)^j$. Also, $P$ is $D$-invariant. 

Suppose $q \equiv -\kappa 1 (\mod 4)$. Then we can find $\beta \in \FF_{q^2}^\times$ of order $(q^2-1)_2$. First we consider the case 
$\kappa = +$. Then $(q^2-1)_2 = (p^2-1)_2$ and so $\beta \in \FF_{p^2} \smallsetminus \FF_q$, $\beta^{p+1} = -1$, and $\beta + \beta^p \in \FF_p$. Now we 
can identify $V_1 = \FF_q^2$ with $\FF_{q^2}$ and take $x$ to be the multiplication by $\beta$, $z$ to be the 
Frobenius map $v \mapsto v^q$. We then choose $(1,\beta)$ to be a basis (over $\FF_q$) for $V$, and use this basis to define $F_p$ and 
$\tau$. This ensures that $F_p$ fixes both $x$ and $z$ (as they have matrices $\begin{pmatrix}0 & 1\\1 & \beta+\beta^p \end{pmatrix}$ and 
$\begin{pmatrix}1 & \beta+\beta^p \\0 & -1 \end{pmatrix}$ in the given basis), and furthermore $\tau(x) = x^{-1}$, $\tau(z) = x^{-2}z$. 
If $\kappa = -$, then we can choose  a Witt basis for $V_1 = \FF_{q^2}^2$, and take 
$$x =  \begin{pmatrix}\beta & 0 \\0 & \beta^{-q} \end{pmatrix},~ z = \begin{pmatrix}0 & 1 \\1 & 0 \end{pmatrix}$$
in this basis. Using the same basis to define $F_p$, we get $F_p(x) = x^p$, $F_p(z) = z$. In either case, $P$ is $D$-invariant, and 
$D$ acts trivially on $P/P' \cong C_2^2$, and again each 
$\lam \in \Irr_{2'}(P)$ is uniquely determined by $\gam \in \Irr(\OB_2(C_\kappa))$ and $j \in \{0,1\}$. 

Now any $\al \in \Irr_{2'}(\NB_G(P))$ is of the form $\al = \delta \otimes \lam$, where $\delta \in \Irr(Z_1)$ and 
$\lam \in \Irr_{2'}(P)$. Then we send $\al$ to $(\hat{s},(2-j,j))$, where $\hat{s} \in \tCk$ corresponds to $(\gam,\delta)$ via \eqref{dual}
and certainly $(2-j,j) \in \HC(2)$. One easily checks that (i) and (ii) hold in this case. 
Furthemore, $\al = \bar\al$ exactly when $\delta=1_{Z_1}$ and $\gam=\bar\gam$, which then 
also imply that $\al$ is rational. It also follows that there are exactly $4$ of such $\al$'s.

\smallskip
(b) In the general case, we can fix a direct sum decomposition 
$$V = \underbrace{V_1 \oplus V_1 \oplus \ldots \oplus V_1}_{2^{m-1}},$$
where $V_1$ is considered with the basis chosen in (a), and fix a basis of $V$ compatible with that basis of $V_1$ and 
this decomposition. We then use this basis to define $F_p$, $\tau$, and 
choose $P = P_1 \rtimes Q$, where 
$P_1 = R \times R \times \ldots \rtimes R \cong R^{2^{m-1}}$, $R \in \Syl_2(GL^\kappa_2(q))$, and $Q \in \Syl_2(\Sigma)$ with
$\Sigma \cong \SSS_{2^{m-1}}$ naturally permuting the $2^{m-1}$ direct summands in $V$ and, correspondingly,
the $2^{m-1}$ direct factors in $P_1$. Note that $Q$ also permutes these
direct factors transitively; furthermore, $D$ acts trivially on $Q$ and stabilizes both $P_1$ and $P$. 
It follows that any $\theta \in \Irr_{2'}(P)$ is uniquely determined by $\lam \in \Irr_{2'}(R)$ and 
$\nu \in \Irr_{2'}(Q)$, where $\theta_{P_1} = \lam \otimes \lam \otimes \ldots \otimes \lam$ and $\theta_Q = \nu$.  
As in (i), $\lam$ corresponds to $(\gam,j)$ with $\gam \in \Irr(\OB_2(C_\kappa))$ and $j \in \{0,1\}$. Each 
$\al \in \Irr_{2'}(\NB_G(P))$ is of the form $\al = \delta \otimes \theta$, where $\delta \in \Irr(Z_1)$ and 
$\theta \in \Irr_{2'}(P)$. By Lemma 3.1 and Theorem 3.2 of \cite{G}, there is a unique hook partition 
$\pi = (2^{m-1}-k,1^k) \vdash 2^{m-1}$ such that $\nu$ is the unique linear constituent of the restriction to $Q$ 
of $\chi^\pi \in \Irr(\SSS_{2^{m-1}})$; in particular, 
$\nu$ is rational. Hence  $\al = \bar\al$ exactly when $\delta = 1_{Z_1}$ and $\lam = \bar\lam$, which, as mentioned in (i), then imply
that $\al$ is rational, and there are exactly $2^{m+1}$ of such $\al$'s. Now we send $\al$ to 
$(\hat{s},\xi)$, where $\hat{s} \in \tCk$ corresponds to $(\gam,\delta)$ via \eqref{dual}, and $\xi \in \HC(n)$ is 
defined as follows 
$$\xi = \left\{ \begin{array}{ll}(2^m-2k,1^{2k}), & j=0\\ (2^m-2k-1,1^{2k+1}), & j = 1. \end{array} \right.$$
It is straightforward to see that the defined map is a bijection between $\Irr_{2'}(\NB_G(P))$ and 
$\tCk \times \HC(n)$, and that both (i) and (ii) hold.
\end{proof}
 
\begin{lem}\label{unitary}
Let $q$ be an odd prime power, $n \geq 2$, $G  = GU_n(q)$, and let  $\chi \in \Irr_{2'}(G)$. Then there is a
unique label of the form \eqref{gl1} for $\chi$, where for $1 \leq i \leq m$ we have that $s_i \in C_-$ are pairwise different, 
$S(s_i,\lam_i) \in \Irr_{2'}(GU_{k_i}(q))$, 
$\chi^{\lam_i} \in \Irr_{2'}(\SSS_{k_i})$,  and
$$[n]_2 = [k_1]_2 < [k_2]_2 < \ldots < [k_m]_2.$$
\end{lem}

\begin{proof}
We can identify the dual group $G^*$ with $G$ and consider the Jordan correspondent $(s,\varphi)$ of $\chi$, where 
$s \in G$ is semisimple and $\varphi$ is a unipotent character of $\CB_G(s)$, cf. \cite{DM}. It is easy to see that the condition $2\nmid \chi(1)$ 
implies that 
$$\CB_G(s) \cong GU_{k_1}(q) \times GU_{k_2}(q) \times \ldots \times GU_{k_m}(q)$$ 
with $\sum^m_{i=1}k_i = n$ and 
furthermore, $n!/\prod^m_{i=1}k_i!$ is odd by \cite[Lemma 4.4(i)]{NT1}.  Hence by Corollary \ref{binom2} there is a unique relabeling of the $k_i$'s such that 
$$[n]_2 = [k_1]_2 < [k_2]_2 < \ldots < [k_m]_2.$$ 
Now we can write $s = \diag(s_1,s_2, \ldots ,s_m)$ with $s_i \in \ZB(GU_{k_i}(q))$ and 
then identify $\ZB(GU_{k_i}(q))$ with $C_-$. Note that the $s_i$'s are pairwise different because of the structure of $\CB_G(s)$.  
Also, we can write 
$\varphi = \varphi_1 \otimes \varphi_2 \otimes \ldots \otimes \varphi_m$ with $\varphi_i = \varphi^{\lam_i} \in \Irr(GU_{k_i}(q))$ being the 
unipotent character labeled by
$\lam_i \vdash k_i$. Since $2 \nmid \varphi(1)$, we see that $\chi^{\lam_i}(1) \equiv \varphi_i(1) \equiv 1 (\mod 2)$ by 
\cite[(1.15)]{FS}. As before, $S(s_i,\lam_i)$ is the irreducible character of $GU_{k_i}(q)$ corresponding to $(s_i,\varphi^{\lam_i})$.
We also note by \cite[Theorem 13.25]{DM} that
\begin{equation}\label{rgl} 
  \chi = R^G_{\CB_G(s)}\left(S(s_1,\lam_1) \otimes S(s_2,\lam_2) \otimes \ldots \otimes \ldots S(s_m,\lam_m)\right),
\end{equation}  
where $R^G_{\CB_G(s)}$ is now the Lusztig induction.
\end{proof} 
 
We can now prove the following result which implies Theorem E:

\begin{thm}\label{main-glu}
Let $n \in \Z_{\geq 1}$, $q$ be an odd prime power, $G = GL_n(q)$ or $GU_n(q)$, and $P \in \Syl_2(G)$. Then 
\begin{enumerate}[\rm (i)]
\item The field of values of any character in $\Irr_{2'}(G)$ and $\Irr_{2'}(\NB_G(P))$ is contained in $\QQ(\tv) = \QQ(\exp(\frac{2\pi i}{q-\kappa 1}))$,
where $\kappa = +$ if $G = GL_n(q)$ and $\kappa = -$ if $G = GU_n(q)$.
\item There is a canonical 
bijection $\chi \mapsto \chi^\sharp$ between $\Irr_{2'}(G)$ and $\Irr_{2'}(\NB_G(P))$ that commutes with the action of 
$\Gamma = {\mathrm {Gal}}(\QQ(\tv)/\QQ)$. Furthermore, we can choose $P$ to be $D$-invariant and 
$\chi \mapsto \chi^\sharp$ to be $D$-equivariant, where $D$ is defined in \eqref{outer}.
\end{enumerate}
\end{thm} 
 
\begin{proof}
Write $n = 2^{n_1} + 2^{n_2} + \ldots + 2^{n_r}$ with $n_1  > \ldots > n_r \geq 0$, and define 
$$\Omega(n) := \tCk \times \HC(2^{n_1}) \times \tCk \times \HC(2^{n_2}) \times \ldots 
    \times \tCk \times \HC(2^{n_r}).$$
We will construct the desired bijection by composing two canonical bijections 
$$\al: \Irr_{2'}(G) \to \Omega(n),~~ \beta:\Omega(n) \to \Irr_{2'}(\NB_G(P)).$$

\smallskip
(a) First we consider the case $r=1$. Then $\beta^{-1}$ is the inverse of the map constructed in Lemma \ref{power}.
Next, if $\chi \in \Irr_{2'}(G)$, then by Theorem \ref{main-gl}(i) and Lemma \ref{unitary}, $\chi = S(s,\lam)$ with $s \in \tCk$ and 
$\lam \vdash n$, and moreover $\chi^\lam \in \Irr(\SSS_n)$ has odd degree by \eqref{gl2}. Hence $\lam \in \HC(n)$ by 
\cite[Lemma 3.1]{G} and we can define $\al(\chi) = (s,\lam)$. Note that, since $G$ is uniform, unipotent characters of $G$ are $\QQ$-linear 
combinations of the Deligne-Lusztig character $R^G_T(1_T)$, $T < G$ any maximal torus, whence they are rational. Furthermore,
$\hat{s} \circ {\mathrm {det}}$ takes values in $\QQ(\tv)$. Hence, $\QQ(\chi) \subseteq \QQ(\tv)$, and for any $\sigma \in \Gamma$, 
$$\chi^\sigma = S(s,\lam)^\sigma = S(s^\sigma,\lam).$$
Next, the unipotent character $\varphi^\lam$ is $D$-invariant (see eg. \cite[Theorem 2.5]{M}), and $F_p$, respectively $\tau$, sends
$\hat{s} \circ {\mathrm {det}}$ to $\hat{s}^p \circ {\mathrm {det}}$, respectively $\hat{s}^{-1} \circ {\mathrm {det}}$.
Hence, (i) and (ii) hold in this case.

\smallskip
(b) In the general case, we can fix a decomposition (orthogonal if $\kappa = -$)
$$V = V_1 \oplus V_2 \oplus \ldots \oplus V_r,$$
where $V_i = \FF_q^{2^{n_i}}$ if $\kappa = +$ and $V_i = \FF_{q^2}^{2^{n_i}}$ if $\kappa = -$. 
We also fix a basis in $V$ compatible with this decomposition and define $F_p$, $\tau$ in this basis.
Then we choose $P = P_1 \times P_2 \times \ldots \times P_r$ with $P_i \in \Syl_2(G_i)$ being $D$-invariant for 
$G_i:= GL^\kappa(V_i) \cong GL^\kappa_{2^{n_i}}(q)$. Note that $P_i$ is an irreducible subgroup of $G_i$ and so
\begin{equation}\label{norm2}
  \NB_G(P) = \NB_{G_1}(P_1) \times \ldots \times \NB_{G_r}(P_r).
\end{equation}   
So if $\theta = \theta_1 \otimes \ldots \otimes \theta_r \in \Irr_{2'}(\NB_G(P))$, then we can define
$$\beta(\theta) := (\hat{t}_1, \mu_1,\hat{t}_2, \mu_2, \ldots,\hat{t}_r,\mu_r)$$
if the bijection in Lemma \ref{power} sends $\theta_i \in \Irr_{2'}(\NB_{G_i}(P_i))$ to $(\hat{t}_i,\mu_i)$. Lemma \ref{power}
also implies that $\QQ(\theta) \subseteq \QQ(\tv)$ and that $\beta$ commutes with the action of $\Gamma$ and $D$.

Now consider any $\chi \in \Irr_{2'}(G)$ and apply Theorem \ref{main-gl}(i) and Lemma \ref{unitary} to $\chi$.  
Assume first that $m = 1$, so that $\chi = S(s,\lam)$. Then $2\nmid \chi^\lam(1)$ by \eqref{gl2}. Applying 
Corollary \ref{young} to the Young subgroup $Y = \SSS_{2^{n_1}} \times \ldots \times \SSS_{2^{n_r}}$, we obtain
$$(\chi^\lam)^\star = \chi^{\nu_1} \otimes \chi^{\nu_2} \otimes \ldots \otimes \chi^{\nu_r},$$
where $\nu_i \in \HC(2^{n_i})$ by \cite[Lemma 3.1]{G}. In this case, we define
\begin{equation}\label{for-a}
  \al(\chi) := (\hat{s},\nu_1, \hat{s},\nu_2, \ldots,\hat{s},\nu_r).
\end{equation}   
In the case of general $m$, note that, by Lemma \ref{binom1}, each 
$k_i$ is the sum of some $2^{n_j}$'s. Moreover, when we express all $k_i$, $1 \leq i \leq m$ this way, each $2^{n_j}$ occurs in precisely one
of these $m$ expressions. Now we can apply \eqref{for-a} to each $S(s_i,\lam_i)$ and then define $\al(\chi)$ by putting all 
$\al(S(s_i,\lam_i))$ together. It is easy to check that the resulting map is a bijection. 
Furthermore, as in (a), $\hat{s_i} \circ {\mathrm {det}}$ take values in $\QQ(\tv)$. Hence, \eqref{gl1}, \eqref{rgl}, and \cite[Proposition 12.2]{DM} (and
the paragraph right before it) imply that $\QQ(\chi) \subseteq \QQ(\tv)$ and that, if 
$$\chi = R^G_L\left(S(s_1,\lam_1) \otimes S(s_2,\lam_2) \otimes \ldots \otimes \ldots S(s_m,\lam_m)\right)$$
(for a suitable Levi subgroup $L$ which in our case can be chosen to be $D$-invariant),
then for any $\sigma \in \Gamma$ we have 
$$\chi^\sigma = R^G_L\left(S(s_1^\sigma,\lam_1) \otimes S(s_2^\sigma,\lam_2) \otimes \ldots \otimes \ldots S(s_m^\sigma,\lam_m)\right).$$
Thus $\al$ commutes with the action of $\Gamma$. The fact that $\al$ commutes with the action of $F_p$ follows from 
\cite[Corollary 2.3]{NTT} and the arguments in (a). In the case $\kappa = +$, $\al$ also commutes with 
$\tau$ as $R^G_L$ is just the Harish-Chandra induction. Hence, (i) and (ii) hold in this case as well. 
\end{proof}

Note that if $P \in \Syl_p(G)$ satisfies $\NB_G(P) = P\CB_G(P)$ for some {\it odd} prime $p$, then the restriction yields 
a natural correspondence between the $p'$-degree irreducible characters of the principal $p$-block of $G$ and those of
$\NB_G(P)$, see \cite[Theorem A]{NTV}. Theorem E yields a canonical correspondence but with $p = 2$.

\begin{proof}[Proof of Corollary F]
The number of real odd-degree irreducible characters of $\NB_G(P)$ can be easily computed using Lemma \ref{power} and 
\eqref{norm2}. Since the correspondence $\chi \mapsto \chi^\sharp$ preserves fields of values of characters by Theorem E,
the statement follows.
\end{proof}
  
\begin{cor}\label{parabolic}
Let $n \in \Z_{>1}$, $q$ be an odd prime power, $G = GL_n(q)$, and let $P$ be a parabolic subgroup of 
odd index in $G$ with Levi subgroup $L$. Then there is a canonical bijection between $\Irr_{2'}(G)$, $\Irr_{2'}(P)$, and 
$\Irr_{2'}(L)$.
\end{cor}

\begin{proof}
We may assume that $P$ is a standard parabolic subgroup with Levi subgroup 
$L = GL_{k_1}(q) \times GL_{k_2}(q) \times \ldots \times GL_{k_m}(q)$, where $m > 1$. As in the proof of Lemma \ref{order},
$2\nmid |G:P|$ implies that we may relabel the $k_i$'s so that $[k_1]_2 < [k_2]_2 < \ldots < [k_m]_2$. 
We also write $n = 2^{n_1} + 2^{n_2} + \ldots + 2^{n_r}$ where $n_1 > n_2 > \ldots > n_r \geq 0$. As in the proof of Theorem E, 
note by Lemma \ref{binom1} that each $k_i$ is the sum of some $2^{n_j}$'s. Moreover, when we express all $k_i$, $1 \leq i \leq m$ this way, 
each $2^{n_j}$ occurs in precisely one of these $m$ expressions.
For each $j$, choose $Q_j \in \Syl_2(GL_{2^{n_j}}(q))$. Then, if 
$k_i = 2^{n_{j_1}} + 2^{n_{j_2}} + \ldots + 2^{n_{j_a}}$,  we can choose $R_i := Q_{j_1} \times Q_{j_2} \times \ldots \times Q_{j_a} \in \Syl_2(GL_{k_i}(q))$,
and  $R := R_1 \times R_2 \times \ldots \times R_m \in \Syl_2(G)$. The formula \eqref{norm2} shows that 
$$\NB_G(R) = \NB_{GL_{k_1}(q)}(R_1) \times \NB_{GL_{k_2}(q)}(R_2) \times \ldots \times \NB_{GL_{k_m}(q)}(R_m)< L.$$ 
Consider any $\chi \in \Irr_{2'}(G)$ and let $\theta := \chi^\sharp \in \Irr_{2'}(\NB_G(R))$ as given by the correspondence in Theorem E.
Next, write $\theta= \theta_1 \otimes \theta_2 \otimes \ldots \otimes \theta_m$ with $\theta_i \in \Irr_{2'}(\NB_{GL_{k_i}(q)}(R_i))$. By 
Theorem E, for each $i$ there is a unique $\psi_i \in \Irr_{2'}(GL_{k_i}(q))$ such that 
$(\psi_i)^\sharp = \theta_i$. Now we can define 
$$\chi^\star = \psi := \psi_1 \otimes \psi_2 \otimes \ldots \otimes \psi_m \in \Irr_{2'}(L)$$
and check that the map $\chi \mapsto \chi^\star$ is a bijection between $\Irr_{2'}(G)$ and $\Irr_{2'}(L)$.

Finally, we show that the inflation (from $L$ to $P$) gives a natural bijection between $\Irr_{2'}(L)$ and $\Irr_{2'}(P)$. For, consider any 
$\varphi \in \Irr_{2'}(P)$ and suppose that $\varphi$ is nontrivial at $U$, the unipotent radical of $P$. Then the $L$-orbit on 
the irreducible constituents of $\varphi_U$ has odd size, and so $R < L$ fixes some $1_U \neq \lam \in \Irr(U)$. On the other
hand, \eqref{norm1} and \eqref{norm2} show that $\CB_G(R) \cap U = 1$. As $R$ acts coprimely on $U$, the Glauberman correspondence
\cite[Theorem 13.1]{Is} implies that $1_U$ is the only $R$-invariant irreducible character of $U$.
\end{proof}

The same proof as above yields an analogue of Corollary \ref{parabolic} for $G = GU_n(q)$ with $2\nmid q$: 

\begin{cor}\label{levi-unitary}
Let $n \in \Z_{>1}$, $q$ be an odd prime power, $G = GU_n(q)$, and let 
$L \cong GU_{k_1}(q) \times GU_{k_2}(q) \times \ldots \times GU_{k_m}(q)$ 
be a Levi subgroup of odd index in $G$. Then there is a canonical bijection between $\Irr_{2'}(G)$ and 
$\Irr_{2'}(L)$.
\hfill $\Box$
\end{cor}


\begin{thebibliography}{ABCD}

\bibitem[APS]{APS}
  A. Ayyer, A. Prasad, and S. Spallone, Odd partitions in Young's lattice, arXiv:1601.01776v1.

\bibitem[C]{C}
  C. W. Curtis, Reduction theorems for characters of finite groups of Lie type,
{\it J. Math. Soc. Jap.} {\bf 27} (1975), 666--688. 
 
\bibitem[DM]{DM}
  F. Digne and J. Michel, `{\it Representations of Finite Groups of
Lie Type}', London Mathematical Society Student Texts {\bf 21}, Cambridge University Press,
1991.

\bibitem[DF]{DF}
 R. Dipper and P. Fleischmann, Modular Harish-Chandra theory I,
{\it Math. Z.} {\bf 211} (1992), 49--71.  


\bibitem[FS]{FS}
  P. Fong and B. Srinivasan, The blocks of finite general and unitary groups,
{\it Invent. Math.} {\bf 69} (1982), 109--153.  
  
\bibitem[G]{G}
  E. Giannelli, Characters of odd degree of symmetric groups, 	arXiv:1601.03839.

\bibitem[HL]{HL}
  R. B. Howlett and G. I. Lehrer, Representations of generic algebras and finite groups of 
Lie type, {\it Trans. Amer. Math. Soc.} {\bf 280} (1983), 753--779.


\bibitem[Is${}_1$]{Isa73}
 I.~M.~Isaacs.
 \newblock Characters of solvable and symplectic groups. 
 \newblock {\em Amer. J. Math.},
 {\bf 95} (1973), 594--635.
 
\bibitem[Is${}_2$]{Is}
I. M. Isaacs, `{\it Character Theory of Finite Groups}',
AMS-Chelsea, Providence, 2006.


\bibitem[J${}_1$]{James}
G.D.~James , \emph{The Representation Theory of the Symmetric
  Groups}, Springer-Verlag, 1978. 
  
\bibitem[J${}_2$]{JamesGL}
G. James, The irreducible representations of the finite general linear groups, {\em Proc. London Math. Soc.} {\bf 52} (1986), 236--268.

\bibitem[JK]{JK}
G.~James and A.~Kerber, \emph{The Representation Theory of the Symmetric Group}, Encyclopedia of Mathematics and its Applications, vol.~16,
  Addison-Wesley Publishing Co., Reading, Mass., 1981. 

\bibitem[K]{K}
  W. M. Kantor,  Primitive permutation groups of odd degree, and an application to finite projective
planes, {\it J. Algebra} {\bf 106} (1987), 15--45.

\bibitem[KT${}_1$]{KT1}
  A. S. Kleshchev and Pham Huu Tiep, Representations of finite special linear groups in non-defining characteristic, 
{\it Adv. Math.} {\bf 220} (2009), 478--504.

\bibitem[KT${}_2$]{KT2}
  A. S. Kleshchev and Pham Huu Tiep, 
Representations of general linear groups which are irreducible over subgroups, 
{\it Amer. J. Math.} {\bf 132} (2010), 425--473.

\bibitem[Ko]{Ko}
  A. S. Kondrat'ev, Normalizers of the Sylow $2$-subgroups in finite simple groups,
{\it Math. Notes}  {\bf 78} (2005), 338--346; translated from {\it Mat. Zametki}  {\bf 78} (2005), 368--376.   

\bibitem[LS]{LS}
  M. W. Liebeck and J. Saxl, The primitive permutation groups of odd degree,
{\it J. London Math. Soc.} {\bf 31} (1985), 250--264.  
  
\bibitem[L]{L}
  F. L\"ubeck, Charaktertafeln f\"ur die Gruppen $CSp_{6}(q)$ mit ungeradem $q$ und
$Sp_{6}(q)$ mit geradem $q$, Preprint 93-61, IWR Heldelberg, 1993.

\bibitem[M]{M}
  G. Malle, Extentions of unipotent characters and the inductive McKay condition, {\it J. Algebra}  {\bf 320} (2008), 2963--2980.

\bibitem[MS]{MS}
  G. Malle and B. Sp\"ath, Characters of odd degree, arXiv:1506.07690v1. 

\bibitem[M${}_1$]{McKay} J. McKay, A new invariant for simple groups, {\em Notices Amer. Math. Soc.} {\bf 18} (1971), 397. 

\bibitem[M${}_2$]{McK}
  J. McKay, Irreducible representations of odd degree, {\it J. Algebra} {\bf 20} (1972), 416--418. 

\bibitem[N${}_1$]{N1}
  G. Navarro, Linear characters of Sylow subgroups, {\it J. Algebra} {\bf 269} (2003), 589--598.
   
\bibitem[N${}_2$]{N2}
  G. Navarro, The McKay conjecture and Galois automorphisms,  {\it Annals of Math.} {\bf 160} (2004), 1129--1140.
    
\bibitem[NT${}_1$]{NT1}
  G. Navarro and Pham Huu Tiep, Irreducible representations of odd degree, {\it Math. Annalen} (to appear).

\bibitem[NT${}_2$]{NT2}
  G. Navarro and Pham Huu Tiep, Real groups and Sylow $2$-subgroups, (submitted).

\bibitem[NTT]{NTT}
  G. Navarro, Pham Huu Tiep, and A. Turull, Brauer characters with cyclotomic field of values, {\it J. Pure Appl. Alg.} 
{\bf  212} (2008), 628--635.

\bibitem[NTV]{NTV}
  G. Navarro, Pham Huu Tiep, and C. Vallejo, McKay natural correspondences on characters, 
{\it Algebra Number Theory} {\bf 8} (2014), 1839--1856.

\bibitem[O1]{O1}
  J. B. Olsson, McKay numbers and heights of characters, 
{\it Math. Scand.} {\bf 38} (1976),  25--42.  

\bibitem[O2]{O2}
  J. B. Olsson, `{\it Combinatorics and Representations of Finite Groups}, vol. {\bf 20},
Vorlesungen aus dem Fachbereich Mathematik der Universit\"at GH Essen,
Universit\"at Essen, 1993.

\bibitem[S]{S}
  L. Sanus, Induction and character correspondences in groups of odd order, {\it J. Algebra} {\bf 249} (2002), 186--195. 
 
\end{thebibliography}
\end{document}